\documentclass[a4paper,11pt]{amsart}
\usepackage{amsmath,amsfonts,amssymb,amsthm,enumerate}
\usepackage[left=2.5cm,right=2.5cm,top=3cm,bottom=3cm,a4paper]{geometry}
\usepackage[onehalfspacing]{setspace}
\usepackage{mathrsfs}
\usepackage[pdftex]{graphicx}
\usepackage[all]{xy}
\usepackage{tikz}
\usepackage{float}
\usepackage{tikz-3dplot}

\usetikzlibrary{positioning}
\tikzset{>=stealth}

\theoremstyle{plain}
\newtheorem{thm}{Theorem}[section]
\newtheorem{lem}[thm]{Lemma}
\newtheorem{prop}[thm]{Proposition}
\newtheorem{cor}[thm]{Corollary}

\theoremstyle{definition}
\newtheorem{defn}[thm]{Definition}
\newtheorem{conj}[thm]{Conjecture} 
 
\newtheorem*{notation}{Notation}
\newtheorem{rmk}[thm]{Remark}
\newtheorem{note}[thm]{Note}
\newtheorem{ques}[thm]{Question}

\newtheorem*{Op}{Organization of the paper}
\newtheorem*{acknowledgements}{Acknowledgements}

\newcommand{\OO}{\mathcal{O}}

\title{Extended Okounkov bodies and multi-point Seshadri constants}
\author{Jaesun Shin}
\date{}


\address{Department of Mathematical Sciences, KAIST, 291 Daehak-ro, Yuseong-gu, Daejeon 305-701, Korea}

\email{jsshin1991@kaist.ac.kr}

\begin{document}
\maketitle

\begin{abstract}
Based on the work of Okounkov (\cite{O1, O2}), Kaveh-Khovanskii (\cite{KK}) and Lazarsfeld-Musta\c t\v a (\cite{LM}) independently associated a convex body, called the Okounkov body, to a big divisor on a normal projective variety with respect to an admissible flag. Although the Okounkov bodies carry rich positivity data of big divisors, they only provide information near a single point. The purpose of this paper is to introduce a convex body of a big divisor that is effective in handling the positivity theory associated with multi-point settings. These convex bodies open the door to approach the local positivity theory at multiple points from a convex-geometric perspective. We study their properties and shapes, and describe local positivity data via them. Finally, we observe the irrationality of Seshadri constants with the help of a relation between Nakayama constants and Seshadri constants.
\end{abstract}

\begin{section} {Introduction}
After the advent of Okounkov bodies in projective geometry, the main question is how to connect them with the geometry of an underlying polarized variety. Thanks to \cite[Proposition 4.1]{LM} and \cite[Theorem A]{J}, it is expected that we should be able to gain information about line bundles in terms of Okounkov bodies. (See \cite{R} for an infinitesimal version in case of surfaces.) In addition to the numerical data (\cite{CHPW, CPW1, CPW2, I, KL1703, KL1507, KL1411, KL1506, KLM}) that Okounkov bodies have, they also provide a useful tool in analyzing higher syzygies on polarized abelian surfaces (\cite{KL1509, S}). 

Despite the usefulness of Okounkov body in projective geometry, it can only provide local positivity data around a point: more precisely, for an admissible flag $Y_{\bullet}:X=Y_{0} \supseteq \dots \supseteq Y_{n}=\{x\}$, the Okounkov body $\Delta_{Y_{\bullet}}(D)$ of a big divisor $D$ has local positivity data near $x$. In other words, it is difficult to address the positivity problems associated with multi-point settings using the Okounkov body. 

The purpose of this paper is to introduce a convex body of a linear series, which we call the extended Okounkov body, to address this problem. As one might expect, this is a generalization of Okounkov bodies and shares many of their useful properties. Moreover, we can approach many positivity data which cannot be handled by the theory of Okounkov bodies with the theory of extended Okounkov bodies. 

We start by constructing the extended Okounkov body. Let $X$ be a smooth projective variety of dimension $n$, and let $Y^{i}_{\bullet}:X=Y^{i}_{0} \supseteq Y^{i}_{1} \supseteq \cdots \supseteq Y^{i}_{n}=\{x_{i}\}$ be admissible flags for $i=1,\dots,r$, where $x_{i} \notin Y^{j}_{1}$ for all $j \neq i$. Given a big line bundle $D$ on $X$, one defines a function 
\begin{align*}
\nu_{Y^{1}_{\bullet}, \dots, Y^{r}_{\bullet}} : \text{ }H^{0}(X,\OO_{X}(D)) \rightarrow \mathbb{Z}^{nr} \cup \{\infty\}, \text{  } s \mapsto (\nu_{1}^{(1)}(s), \dots, \nu_{n}^{(1)}(s) \text{ ; } \dots \text{ ; } \nu_{1}^{(r)}(s), \dots, \nu_{n}^{(r)}(s))
\end{align*}
as follows. First, set $\nu_{1}^{(i)}(s)={\rm ord}_{Y^{i}_{1}}(s).$ After choosing local equations for $Y^{i}_{1}$'s  in $X$, $s$ determines a section 
\begin{align*}
\bar{s}_{1} \in H^{0}(X,\OO_{X}(D-\nu_{1}^{(1)}(s)Y^{1}_{1}- \cdots -\nu^{(r)}_{1}(s)Y^{r}_{1})).
\end{align*}
By restricting $\bar{s}_{1}$ to each $Y^{i}_{1}$, we get $r$ sections 
\begin{align*}
s^{(i)}_{1} \in H^{0}(Y_{1}^{i},\OO_{Y_{1}^{i}}({(D-\nu_{1}^{(1)}(s)Y^{1}_{1}- \cdots -\nu^{(r)}_{1}(s)Y^{r}_{1})|}_{Y_{1}^{i}})).
\end{align*}
For each $s_{1}^{(i)}$, we take $\nu_{2}^{(i)}(s)={\rm ord}_{Y_{2}^{i}}(s_{1}^{(i)})$ and continue in this manner to define $\nu_{Y^{1}_{\bullet}, \dots, Y^{r}_{\bullet}}(s)$. 

Next, define the semi-group
\begin{align*}
\Gamma_{Y^{1}_{\bullet}, \dots, Y^{r}_{\bullet}}(D)=\{(\nu_{Y^{1}_{\bullet}, \dots, Y^{r}_{\bullet}}(s),m) \text{ }|\text{ } 0 \neq s \in H^{0}(X,\OO_{X}(mD)), \text{ } m \in \mathbb{N}\} 
\end{align*}
of $\mathbb{N}^{nr} \times \mathbb{N} \subseteq \mathbb{R}^{nr} \times \mathbb{R}$. Then the extended Okounkov body of $D$ with respect to $Y^{1}_{\bullet}, \dots, Y^{r}_{\bullet}$ is the convex body 
\begin{align*}
\Delta_{Y^{1}_{\bullet}, \dots, Y^{r}_{\bullet}}(D)={\rm Cone}(\Gamma_{Y^{1}_{\bullet}, \dots, Y^{r}_{\bullet}}(D)) \cap (\mathbb{R}^{nr} \times \{1\}).
\end{align*}

As a first step, we study the variation of these bodies as functions of big divisors. It is easy to check that $\Delta_{Y^{1}_{\bullet}, \dots, Y^{r}_{\bullet}}(pD)=p \cdot \Delta_{Y^{1}_{\bullet}, \dots, Y^{r}_{\bullet}}(D)$ for every integer $p$, so there is a naturally defined $\Delta_{Y^{1}_{\bullet}, \dots, Y^{r}_{\bullet}}(D)$ for any big rational class $D$. In the theory of Okounkov bodies, there exists a closed convex cone that observes the variation of Okounkov bodies, which is called the global Okounkov body (\cite[Theorem B]{LM}). The existence of such a closed convex cone implies that each Okounkov body fits together nicely: they vary continuously in the cone of big $\mathbb{R}$-divisors. Moreover, this yields a natural definition of $\Delta_{Y^{1}_{\bullet},\dots, Y^{r}_{\bullet}}(D)$ for any big $\mathbb{R}$-divisor $D$ on $X$ (cf. Remark \ref{rmk:R-divisor}). As one might expect, a similar situation occurs in the theory of extended Okounkov bodies:

\begin{thm} \label{thm:global}
Let $X$ be a normal projective variety of dimension $n$, and let $Y^{1}_{\bullet}, \dots, Y^{r}_{\bullet}$ be admissible flags centered at $x_{1}, \dots, x_{r} \in X$. Then there exists a closed convex cone $\Delta_{Y^{1}_{\bullet},\dots, Y^{r}_{\bullet}}(X) \subseteq \mathbb{R}^{nr} \times {{\rm N^{1}}(X)}_{\mathbb{R}}$ characterized by the property that in the commutative diagram 
\begin{displaymath}
\xymatrix{
\Delta_{Y^{1}_{\bullet},\dots, Y^{r}_{\bullet}}(X) \ar[dr] \ar@{^{(}->}[rr] && \mathbb{R}^{nr} \times {{\rm N^{1}}(X)}_{\mathbb{R}} \ar[dl]^{{\rm pr}_{2}} \\
&{{\rm N^{1}}(X)}_{\mathbb{R}}, &}
\end{displaymath}
${\rm pr}_{2}^{-1}(D) \cap \Delta_{Y^{1}_{\bullet},\dots, Y^{r}_{\bullet}}(X)=\Delta_{Y^{1}_{\bullet},\dots, Y^{r}_{\bullet}}(D)$ for any big class $D \in {{\rm N^{1}}(X)}_{\mathbb{Q}}$. 
\end{thm}

Next, we connect Seshadri constants at multiple points with these convex bodies. After Seshadri's criterion for ampleness (\cite[Theorem 7.1]{H}), one tried to measure the extent of its positivity. This leads to the definition of Seshadri constant, introduced by Demailly \cite{JPD}, which measures the local positivity of an ample line bundle at a point (\cite{BDHKKSS}), and its extension, the moving Seshadri constant, was introduced by Nakamaye (\cite{Na}). For this invariant, K\" uronya and Lozovanu presented a nice convex geometric description (\cite[Theorem C]{KL1507}). Motivated by their result, it is natural to ask what can be stated about the Seshadri constants at multiple points (Definition \ref{defn:multi-weight}) in terms of convex geometry. 

To fix terminology, let $x_{1},\dots,x_{r}$ be $r$ distinct points on a smooth projective variety $X$, and denote by $\pi:\tilde{X}={\rm Bl}_{\{x_{1},\dots,x_{r}\}}(X) \rightarrow X$ the blow-up of $X$ at $x_{1},\dots,x_{r}$. When $Y^{1}_{\bullet},\dots,Y^{r}_{\bullet}$ are infinitesimal over $x_{1},\dots,x_{r}$ (\cite[Definition 2.1]{KL1507}), $\Delta_{Y^{1}_{\bullet}, \dots, Y^{r}_{\bullet}}(\pi^{*}D)$ will be denoted by $\widetilde{\Delta}_{Y^{1}_{\bullet}, \dots, Y^{r}_{\bullet}}(D)$. 

Moreover, let $(\nu_{1}^{(1)}, \dots, \nu_{n}^{(1)} \text{ ; } \dots \text{ ; } \nu_{1}^{(r)}, \dots, \nu_{n}^{(r)})$ be the standard coordinate of $\mathbb{R}^{nr}$. For a $(\xi_{1},\dots,\xi_{r}) \in \mathbb{R}^{r}$, denote by ${\bf v}^{(\xi_{1},\dots,\xi_{r})}_{j}=\sum_{i=1}^{r} \xi_{i} \cdot \nu_{j}^{(i)}$. Finally, let us define the inverted standard slice simplex of size $(\xi_{1},\dots,\xi_{r}) \in \mathbb{R}_{\ge 0}^{r}$ in $\mathbb{R}^{nr}$: this is the convex body
\begin{align*}
\Delta_{(\xi_{1},\dots,\xi_{r})}^{-r} \overset{\rm def}{=} \text{convex hull of } \{{\bf 0}, \text{ } {\bf v}^{(\xi_{1},\dots,\xi_{r})}_{1}, {\bf v}^{(\xi_{1},\dots,\xi_{r})}_{1}+{\bf v}^{(\xi_{1},\dots,\xi_{r})}_{2}, \dots, \text{ } {\bf v}^{(\xi_{1},\dots,\xi_{r})}_{1}+{\bf v}^{(\xi_{1},\dots,\xi_{r})}_{n}\} \subseteq \mathbb{R}^{nr}.
\end{align*}

It follows from our argument that infinitesimal extended Okounkov bodies over points where the big $\mathbb{R}$-divisor $D$ is locally ample always contain the inverted standard slice simplices of size $(m_{1}a,\dots,m_{r}a)$ for some $a$, depending on $m_{1},\dots,m_{r}$. For all such infinitesimal flags $Y^{1}_{\bullet}, \dots, Y^{r}_{\bullet}$, the supremum of $a$ satisfying $ \Delta_{(m_{1}a,\dots,m_{r}a)}^{-r} \subseteq \widetilde{\Delta}_{Y^{1}_{\bullet}, \dots, Y^{r}_{\bullet}}(D)$ is called the largest inverted slice simplex constant with multi-weight ${\bf m}=(m_{1},\dots,m_{r})$ and will be denoted by $\xi_{\bf m}(D;x_{1},\dots,x_{r})$. As a result of our efforts, we obtain a description of multi-weight moving Seshadri constants in the following form. 

\begin{thm} \label{thm:multi-point Seshadri constants}
Let $D$ be a big $\mathbb{R}$-divisor on $X$. Then
\begin{align*}
\epsilon_{\bf m}(\lVert D \rVert;x_{1},\dots,x_{r})=\xi_{\bf m}(D;x_{1},\dots,x_{r})
\end{align*}
for any ${\bf m} \in \mathbb{N}^{r}$ and any $x_{1},\dots,x_{r} \in X$.
\end{thm}

One of the most important aspects of the Okounkov bodies is that they encode interesection theory of $D$ as Euclidean volumes of these convex sets. In this regard, we propose the following conjecture and obtain a partial answer. Before we proceed, we denote by $S_{(m_{1},\dots,m_{r})}$ the slice
\begin{align*}
\{(\nu_{1}^{(1)}, \dots, \nu_{n}^{(1)} \text{ ; } \dots \text{ ; } \nu_{1}^{(r)}, \dots, \nu_{n}^{(r)}) \in \mathbb{R}^{rn} \text{ } \lvert \text{ } \frac{\nu_{1}^{(1)}}{m_{1}}=\cdots=\frac{\nu_{1}^{(r)}}{m_{r}}, \dots, \text{ } \frac{\nu_{n}^{(1)}}{m_{1}}=\cdots=\frac{\nu_{n}^{(r)}}{m_{r}}\} \cong \mathbb{R}^{n}. 
\end{align*}

\begin{conj} \label{conj:Shin}
Let $\pi:{\rm Bl}_{r}(X) \rightarrow X$ be the blow-up of $X$ at $r$ general points $x_{1},\dots,x_{r}$, and let $D$ be a big $\mathbb{R}$-divisor on $X$. Then   
\begin{align*}
{\rm vol}_{\mathbb{R}^{n}}({\widetilde{\Delta}_{Y^{1}_{\bullet},\dots,Y^{r}_{\bullet}}(D)}_{\frac{\nu_{1}^{(1)}}{m_{1}}=\cdots=\frac{\nu_{1}^{(r)}}{m_{r}}, \dots, \text{ } \frac{\nu_{n}^{(1)}}{m_{1}}=\cdots=\frac{\nu_{n}^{(r)}}{m_{r}}})=\frac{(\sqrt{r})^{n-2}}{n!} \cdot {\rm vol}_{X}(D)
\end{align*} 
for all infinitesimal flags $Y^{1}_{\bullet},\dots,Y^{r}_{\bullet}$ over $x_{1},\dots,x_{r}$ and any $(m_{1},\dots,m_{r}) \in \mathbb{N}^{r}$, where $\mathbb{R}^{n}=S_{(m_{1},\dots,m_{r})}$. 
\end{conj}

For the mono-graded case, we have a positive answer to Conjecture \ref{conj:Shin} (Proposition \ref{prop:mono-graded}). As a result, we obtain a relation between $\epsilon(L;x_{1},\dots,x_{r})$ and the Nakayama constant $\mu(L;x_{1},\dots,x_{r})$ for a polarized surface $(S,L)$ (Proposition \ref{prop:general relation}). 

An interesting by-product of Proposition \ref{prop:general relation} is about the irrationality of Seshadri constants on general rational surfaces. In \cite{DKMS}, Dumnick, K\"uronya, Maclean, and Szemberg show for $s \ge 9$, the SHGH conjecture implies the existence of an ample line bundle on ${\rm Bl}_{r}(\mathbb{P}^{2})$ with irrational Seshadri constant. Motivated by their result, Hanumanthu and Harbourne (\cite{HH}) show that the $(-1)$-curve conjecture which is weaker than assuming the SHGH conjecture is sufficient to draw the same (or even a stronger) conclusion. From Proposition \ref{prop:general relation}, we generalize their results (Theorem \ref{thm:irrationality}): under a weaker assumption (Conjecture \ref{conj:standard}, Lemma \ref{lem:(-1) and standard}), we draw a stronger conclusion than \cite{DKMS, HH}. An important point of our proof is that even the full Conjecture \ref{conj:standard} is not needed for the irrationality of Seshadri constants, e.g. Corollary \ref{cor:homogeneous}. 

\begin{notation} We work over the complex numbers, and let $\mathbb{N}=\{1,2, \dots\}$. Denote by ${\bf 0}$ the origin in an Euclidean space $\mathbb{R}^{k}$ for some $k \in \mathbb{N}$. We denote $X$ a smooth projective variety of dimension $n \ge 2$ unless specified. A divisor means a $\mathbb{Q}$-Cartier $\mathbb{Q}$-divisor. For a subset $\Delta \subseteq \mathbb{R}^{rn}$, ${\rm Conv}(\Delta)$ is the smallest closed convex set containing $\Delta$. 

Moreover, let $(\nu_{1}^{(1)}, \dots, \nu_{n}^{(1)} \text{ ; } \dots \text{ ; } \nu_{1}^{(r)}, \dots, \nu_{n}^{(r)})$ be the standard coordinate of $\mathbb{R}^{nr}$. By abuse of notation, we also denote by $\nu_{j}^{(i)}$ the $((i-1)n+j)$-th standard basis vector of $\mathbb{R}^{nr}$ for each $i$, $j$. For a $(\xi_{1},\dots,\xi_{r}) \in \mathbb{R}^{r}$, denote by ${\bf v}^{(\xi_{1},\dots,\xi_{r})}_{j}=\sum_{i=1}^{r} \xi_{i} \cdot \nu_{j}^{(i)}$. Finally, for $(m_{1},\dots,m_{r}) \in \mathbb{R}^{r}$, $S_{(m_{1},\dots,m_{r})}$ denotes the slice
\begin{align*}
\{(\nu_{1}^{(1)}, \dots, \nu_{n}^{(1)} \text{ ; } \dots \text{ ; } \nu_{1}^{(r)}, \dots, \nu_{n}^{(r)}) \in \mathbb{R}^{rn} \text{ } \lvert \text{ } \frac{\nu_{1}^{(1)}}{m_{1}}=\cdots=\frac{\nu_{1}^{(r)}}{m_{r}}, \dots, \text{ } \frac{\nu_{n}^{(1)}}{m_{1}}=\cdots=\frac{\nu_{n}^{(r)}}{m_{r}}\} \cong \mathbb{R}^{n}. 
\end{align*}
\end{notation}

\begin{Op} Concerning the organization of the paper, we begin in Section \ref{section:2} by defining the extended Okounkov bodies of big divisors. We observe their basic properties and their relationship with the Okounkov bodies. Section \ref{section:3} revolves around the variational theory of extended Okounkov bodies. Section \ref{section:4} is devoted to examples. We treat the case of curves, and provide possible descriptions of extended Okounkov bodies of big divisors on surfaces and on toric varieties. Section \ref{section:5} is the main part of this paper: the characterization of asymptotic base loci and the description of multi-weight moving Seshadri constants in terms of extended Okounkov bodies are given. Lastly, in Section \ref{section:6}, we observe volumes of slices of the extended infinitesimal Okounkov bodies, the relation between Nakayama constants and Seshadri constants, and the irrationality of Seshadri constants. 
\end{Op}

\begin{acknowledgements}
I want to express my gratitude to my advisor Yongnam Lee for his advice, encouragement and teaching. I also wish to thank Joaquim Ro\'e for his comments on Nagata's conjecture and multi-weight Seshadri constants, Atsushi Ito for correcting some of my misunderstandings about extended Okounkov bodies, and Dong-Hwi Seo for helpful discussions on proving Proposition \ref{prop:mono-graded}. This work was supported by NRF(National Research Foundation of Korea) Grant funded by the Korean Government(NRF-2016-Fostering Core Leaders of the Future Basic Science Program/Global Ph.D. Fellowship Program). 
\end{acknowledgements}

\end{section}

\begin{section} {Construction of the extended Okounkov body} \label{section:2}
Let $X$ be a normal projective variety of dimension $n$, and let $Y^{i}_{\bullet}:X=Y^{i}_{0} \supseteq Y^{i}_{1} \supseteq \cdots \supseteq Y^{i}_{n}=\{x_{i}\}$ admissible flags centered at $x_{i}$ with $r$ distinct points $x_{1}, \dots, x_{r}$. Furthermore, we always assume that $x_{i} \notin Y^{j}_{1}$ for each $i=1, \dots, r$ and all $j \neq i$. We begin by defining a function $\nu_{Y^{1}_{\bullet}, \dots, Y^{r}_{\bullet}}$ with respect to $Y^{1}_{\bullet}, \dots, Y^{r}_{\bullet}$. 

\begin{defn}
Let $D$ be a big line bundle on $X$. Given $0 \neq s \in H^{0}(X,\OO_{X}(D))$, we denote 
\begin{align*}
\nu_{1}^{(i)}(s)={\rm ord}_{Y^{i}_{1}}(s).
\end{align*}
After choosing local equations for $Y^{i}_{1}$'s  in $X$, $s$ determines a section 
\begin{align*}
\bar{s}_{1} \in H^{0}(X,\OO_{X}(D-\nu_{1}^{(1)}(s)Y^{1}_{1}- \cdots -\nu^{(r)}_{1}(s)Y^{r}_{1})).
\end{align*}
By restricting $\bar{s}_{1}$ to each $Y^{i}_{1}$, we get $r$ sections 
\begin{align*}
s^{(i)}_{1} \in H^{0}(Y_{1}^{i},\OO_{Y_{1}^{i}}({(D-\nu_{1}^{(1)}(s)Y^{1}_{1}- \cdots -\nu^{(r)}_{1}(s)Y^{r}_{1})|}_{Y_{1}^{i}})).
\end{align*}
For each $s_{1}^{(i)}$, we take  
\begin{align*}
\nu_{2}^{(i)}(s)={\rm ord}_{Y_{2}^{i}}(s_{1}^{(i)}).
\end{align*}
Repeating this process, we define a function $\nu_{Y^{1}_{\bullet}, \dots, Y^{r}_{\bullet}}$ with respect to $Y^{1}_{\bullet}, \dots, Y^{r}_{\bullet}$ as follows:
\begin{align*}
\nu_{Y^{1}_{\bullet}, \dots, Y^{r}_{\bullet}} : \text{ }H^{0}(X,\OO_{X}(D)) \rightarrow \mathbb{Z}^{nr} \cup \{\infty\}, \text{  } s \mapsto (\nu_{1}^{(1)}(s), \dots, \nu_{n}^{(1)}(s) \text{ ; } \dots \text{ ; } \nu_{1}^{(r)}(s), \dots, \nu_{n}^{(r)}(s)).
\end{align*}
\end{defn}

\begin{rmk} \label{rmk:separate}
Since $x_{i} \notin Y^{j}_{1}$ for each $i=1, \dots, r$ and all $j \neq i$, we may suppose that each $Y^{1}_{i}$ is mutually disjoint after replacing $X$ by an open set containing all of the $x_{i}$'s. Thus we have
\begin{align*}
H^{0}(Y_{1}^{i},\OO_{Y_{1}^{i}}({(D-\nu_{1}^{(1)}(s)Y^{1}_{1}- \cdots -\nu^{(r)}_{1}(s)Y^{r}_{1})|}_{Y_{1}^{i}}))=H^{0}(Y_{1}^{i},\OO_{Y_{1}^{i}}({(D-\nu_{1}^{(i)}(s)Y^{i}_{1})|}_{Y_{1}^{i}}))
\end{align*}
for all $i=1, \dots, r$ so that $\nu_{Y^{1}_{\bullet}, \dots, Y^{r}_{\bullet}}(s)=\nu_{Y^{1}_{\bullet}}(s) \times \dots \times \nu_{Y^{1}_{\bullet}}(s)$ for all $0 \neq s \in H^{0}(X,\OO_{X}(D))$, where $\nu_{Y^{i}_{\bullet}}$ is the valuation-like function attached to $Y^{i}_{\bullet}$ (\cite[Subsection 1.1]{LM}). 
\end{rmk}

\begin{note}
By the construction, $\nu_{Y^{1}_{\bullet}, \dots, Y^{r}_{\bullet}}$ shares the following properties:
\begin{enumerate}[(i)]
\item $\nu_{Y^{1}_{\bullet}, \dots, Y^{r}_{\bullet}}(s)=\infty$ if and only if $s=0$. 
\item Given non-zero sections $s \in H^{0}(X,\OO_{X}(D))$ and $t \in H^{0}(X, \OO_{X}(E))$, 
\begin{align*}
\nu_{Y^{1}_{\bullet}, \dots, Y^{r}_{\bullet}}(s \otimes t)=\nu_{Y^{1}_{\bullet}, \dots, Y^{r}_{\bullet}}(s)+\nu_{Y^{1}_{\bullet}, \dots, Y^{r}_{\bullet}}(t).
\end{align*} 
\end{enumerate}
\end{note}

For a subset $\Sigma \subseteq \mathbb{R}^{m}$, ${\rm Cone}(\Sigma) \subseteq \mathbb{R}^{m}$ is the closed convex cone spanned by $\Sigma$. The extended Okounkov body of $D$ is defined as the base of some closed convex cone:

\begin{defn}
\begin{enumerate}[(1)]
\item The extended graded semigroup of $D$ is the sub-semigroup 
\begin{align*}
\Gamma_{Y^{1}_{\bullet}, \dots, Y^{r}_{\bullet}}(D)=\{(\nu_{Y^{1}_{\bullet}, \dots, Y^{r}_{\bullet}}(s),m) \text{ }|\text{ } 0 \neq s \in H^{0}(X,\OO_{X}(mD)), \text{ } m \in \mathbb{N}\} 
\end{align*}
of $\mathbb{N}^{nr} \times \mathbb{N} \subseteq \mathbb{R}^{nr} \times \mathbb{R}$.
\item The extended Okounkov body of $D$ with respect to $Y^{1}_{\bullet}, \dots, Y^{r}_{\bullet}$ is the convex body 
\begin{align*}
\Delta_{Y^{1}_{\bullet}, \dots, Y^{r}_{\bullet}}(D)={\rm Cone}(\Gamma_{Y^{1}_{\bullet}, \dots, Y^{r}_{\bullet}}(D)) \cap (\mathbb{R}^{nr} \times \{1\}).
\end{align*}
\end{enumerate}
\end{defn}

As in \cite[Proposition 4.1]{LM} and \cite{J}, one has the numerical invariance and the homogeneity of $\Delta_{Y^{1}_{\bullet}, \dots, Y^{r}_{\bullet}}(-)$ immediately. This yields a natural  definition of $\Delta_{Y^{1}_{\bullet}, \dots, Y^{r}_{\bullet}}(D)$ for any big rational class $D$ on $X$, viz. $\Delta_{Y^{1}_{\bullet}, \dots, Y^{r}_{\bullet}}(D)=\frac{1}{m}\Delta_{Y^{1}_{\bullet}, \dots, Y^{r}_{\bullet}}(mD)$, where $mD$ is integral for some $m \in \mathbb{N}$.




One of the basic properties of $\Delta_{Y^{1}_{\bullet}, \dots, Y^{r}_{\bullet}}(D)$ is that the Okounkov bodies $\Delta_{Y^{i}_{\bullet}}(D)$ are appeared by its projections. This is the reason why we impose the condition that  $x_{i} \notin Y^{j}_{1}$ for each $i=1, \dots, r$ and all $j \neq i$. 

\begin{note} \label{note:projection}
\begin{enumerate}[(1)]
\item Let ${\rm pr}_{i}:\mathbb{R}^{nr} \rightarrow \mathbb{R}^{n}, \text { } (\nu_{1}^{(1)}, \dots, \nu_{n}^{(1)} \text{ ; } \dots \text{ ; } \nu_{1}^{(r)}, \dots \nu_{n}^{(r)}) \mapsto (\nu_{1}^{(i)}, \dots, \nu_{n}^{(i)})$ be a projection map for each $i=1, \dots, r$. Then $\Delta_{Y^{i}_{\bullet}}(D)={\rm pr}_{i}(\Delta_{Y^{1}_{\bullet}, \dots, Y^{r}_{\bullet}}(D))$ for all $i$. 
\item For each $i=1, \dots, r$, let $\Delta_{i}(D)=(0, \dots, 0) \times \dots \times \Delta_{Y^{i}_{\bullet}}(D) \times \dots \times (0, \dots, 0) \subset \mathbb{R}^{nr}$ for admissible flags $Y^{i}_{\bullet}$, where $(0, \dots, 0)$ is an $n$-tuple of zeroes. Then 
\begin{align*}
\Delta_{Y^{1}_{\bullet}, \dots, Y^{r}_{\bullet}}(D) \subseteq \Delta_{1}(D)+\dots +\Delta_{r}(D) \subseteq r \cdot {\rm Conv}(\bigcup_{i=1}^{r}\Delta_{i}(D)).
\end{align*} 
\end{enumerate}
\end{note}

From Note \ref{note:projection}-(1), it is natural to ask whehter 
\begin{align*}
\Delta_{i}(D)=(0, \dots, 0) \times \dots \times \Delta_{Y^{i}_{\bullet}}(D) \times \dots \times (0, \dots, 0) \subseteq \mathbb{R}^{nr}
\end{align*}
is contained in $\Delta_{Y^{1}_{\bullet}, \dots, Y^{r}_{\bullet}}(D)$ when $x_{1},\dots,x_{r} \notin {\rm \bf B}_{+}(D)$:

\begin{ques} \label{ques:general inclusion}
Let $D$ be a big $\mathbb{R}$-divisor on $X$, and let $Y^{1}_{\bullet}, \dots, Y^{r}_{\bullet}$ be admissible flags centered at $x_{1},\dots,x_{r} \notin {\rm \bf B}_{+}(D)$. For each $i=1,\dots,r$, write
\begin{align*}
\Delta_{i}(D)=(0, \dots, 0) \times \dots \times \Delta_{Y^{i}_{\bullet}}(D) \times \dots \times (0, \dots, 0) \subseteq \mathbb{R}^{nr}.
\end{align*}
Then, would the inclusion
\begin{align*}
\Delta_{i}(D) \subseteq \Delta_{Y^{1}_{\bullet},\dots,Y^{r}_{\bullet}}(D)
\end{align*}
be established for all $i=1,\dots,r$?
\end{ques}

See Section \ref{section:3} for the definition of $\Delta_{Y^{1}_{\bullet},\dots,Y^{r}_{\bullet}}(D)$ of a big $\mathbb{R}$-divisor $D$. When $x_{1},\dots,x_{r}$ are very general, we have a positive answer.  

\begin{prop} \label{prop:inclusion 2}
For a big $\mathbb{R}$-divisor $D$ on $X$ and any admissible flags $Y^{1}_{\bullet},\dots,Y^{r}_{\bullet}$ centered at very general points $x_{1},\dots,x_{r} \in X$, we have
\begin{align*}
{\rm Conv}(\bigcup_{i=1}^{r}\Delta_{i}(D)) \subseteq \Delta_{Y^{1}_{\bullet},\dots,Y^{r}_{\bullet}}(D).
\end{align*}
\end{prop}

\begin{proof}
First, we assume that $D$ is a big $\mathbb{Q}$-divisor. Fix infintesimal flags $Y^{1}_{\bullet}, \dots, Y^{r}_{\bullet}$ centered $x_{1},\dots,x_{r}$. \cite[Corollary 1.2]{D} or \cite[Lemma 1.8]{KL1507} gives that all interior points of $\Delta_{Y^{i}_{\bullet}}(D) \cap \mathbb{Q}^{n}$ are valuative (\cite[Definition 1.7]{KL1507}). For each $i=1,\dots,r$, there are countably many sections $s^{(j)}_{i} \in H^{0}(\tilde{X},\OO_{\tilde{X}}(m_{j}D))$ for some $m_{j} \in \mathbb{N}$, depending on $j \in \mathbb{N}$ whose valuations along $Y^{i}_{\bullet}$ form all of the interior points of $\Delta_{Y^{i}_{\bullet}}(D) \cap \mathbb{Q}^{n}$. Since $x_{1},\dots,x_{r}$ are very general, we may assume that for each $i=1, \dots, r$ and all $j \in \mathbb{N}$, the zero loci of $s_{i}^{(j)}$ do not contain all of the $x_{k}$'s with $k \neq i$, that is, for each $i=1,\dots, r$, we have $\Delta_{i}(D) \subseteq \Delta_{Y^{1}_{\bullet},\dots,Y^{r}_{\bullet}}(D)$. Since $\Delta_{Y^{1}_{\bullet},\dots,Y^{r}_{\bullet}}(D)$ is convex, the conclusion holds for big $\mathbb{Q}$-divisors. 

The big $\mathbb{R}$-divisor cases follow immediately by the above arguments and the continuity of extended Okounkov bodies (Theorem \ref{thm:global}). 
\end{proof}

\end{section}

\begin{section} {Variation of extended Okounkov bodies} \label{section:3}
This section is devoted to the proof of Theorem \ref{thm:global}: we construct a closed convex cone 
\begin{align*}
\Delta_{Y^{1}_{\bullet},\dots, Y^{r}_{\bullet}}(X) \subseteq \mathbb{R}^{nr} \times {{\rm N^{1}}(X)}_{\mathbb{R}}
\end{align*}
such that the fiber of $\Delta_{Y^{1}_{\bullet},\dots, Y^{r}_{\bullet}}(X)$ over any big rational class $D$ on $X$ is $\Delta_{Y^{1}_{\bullet},\dots, Y^{r}_{\bullet}}(D)$. 

\begin{proof}[Proof of Theorem \ref{thm:global}]
Let $\rho(X)={\rm dim}_{\mathbb{R}}{{\rm N^{1}}(X)}_{\mathbb{R}}$, and fix divisors $D_{1}, \dots, D_{\rho(X)}$ on $X$ whose classes form a $\mathbb{Z}^{\rho(X)}$-basis of ${\rm N^{1}}(X)$ and $\overline{\rm Eff}(X)$ lies in the positive orthant of $\mathbb{R}^{\rho(X)}$. We define the multigraded semigroup of $X$ with respect to $Y^{1}_{\bullet}, \dots, Y^{r}_{\bullet}$ to be the additive sub-semigroup of $\mathbb{N}^{nr} \times \mathbb{N}^{\rho(X)}$ given by
\begin{align*}
\Gamma_{Y^{1}_{\bullet},\dots, Y^{r}_{\bullet}}(X;D_{1},\dots,D_{\rho(X)}) \overset{\rm def}{=} \{&(\nu_{Y^{1}_{\bullet},\dots, Y^{r}_{\bullet}}(s), a_{1},\dots,a_{\rho(X)}) \in \mathbb{N}^{nr} \times \mathbb{N}^{\rho(X)} \text{ } |\\
&0 \neq s \in H^{0}(X,\OO_{X}(a_{1}D_{1}+\dots+a_{\rho(X)}D_{\rho(X)}))\}.
\end{align*}
Then we simply take 
\begin{align*}
\Delta_{Y^{1}_{\bullet},\dots, Y^{r}_{\bullet}}(X) \overset{\rm def}{=} {\rm Cone}(\Gamma_{Y^{1}_{\bullet},\dots, Y^{r}_{\bullet}}(X;D_{1},\dots,D_{\rho(X)})) \subseteq \mathbb{R}^{nr} \times {{\rm N^{1}}(X)}_{\mathbb{R}}.
\end{align*}
By its construction, everything is clear except that ${\rm pr}_{2}^{-1}(D) \cap \Delta_{Y^{1}_{\bullet},\dots, Y^{r}_{\bullet}}(X)=\Delta_{Y^{1}_{\bullet},\dots, Y^{r}_{\bullet}}(D)$ for any big $\mathbb{Q}$-divisor $D$ on $X$. First, we focus on the case when $D$ is integral. The main ingredient of the proof is \cite[Proposition 4.9]{LM}. 

We claim that $\Gamma(X)=\Gamma(X;D_{1},\dots,D_{\rho(X)})$ generates $\mathbb{Z}^{nr}\times \mathbb{Z}^{\rho(X)}$ as a group. Since $D_{1}, \dots, D_{\rho(X)}$ generate $\mathbb{Z}^{\rho(X)}$, it is enough to show that for any big divisor $D=D_{i}$ with $i=1, \dots, \rho(X)$, $\Gamma_{Y^{1}_{\bullet},\dots, Y^{r}_{\bullet}}(D) \subset \mathbb{N}^{nr+1}$ generates $\mathbb{Z}^{nr+1}$ as a group. Fix such a big divisor $D=D_{i}$. As in the proof of \cite[Lemma 2.2]{LM}, we can write $D=A_{1}-B_{1}$ as the difference of two very ample divisors such that (by adding a further very ample divisor to both $A_{1}$ and $B_{1}$, if necessary)  
\begin{align*}
\nu_{Y^{1}_{\bullet},\dots, Y^{r}_{\bullet}}(s^{(1)}_{0})=\nu_{Y^{1}_{\bullet},\dots, Y^{r}_{\bullet}}(t^{(1)}_{0})&=(0,\dots,0 \text{ ; } \dots \text{ ; } 0, \dots, 0) \in \mathbb{R}^{nr}, \\
(\nu_{1}^{(1)}(t_{i}^{(1)}), \dots, \nu_{n}^{(1)}(t_{i}^{(1)}))&={\bf e}_{i} \in \mathbb{R}^{n}
\end{align*}
for $s^{(1)}_{0} \in H^{0}(X,\OO_{X}(A_{1}))$ and $t^{(1)}_{i} \in H^{0}(X,\OO_{X}(B_{1}))$ with $i=0, \dots, n$, where ${\bf e}_{i} \in \mathbb{Z}^{n}$ is the $i$-th standard basis vector and $\nu_{j}^{(1)}(-)$ is as in the definition of $\nu_{Y^{1}_{\bullet},\dots, Y^{r}_{\bullet}}(-)$. Since $A_{1}, B_{1}$ are sufficiently positive very ample divisors, there are sufficiently many sections satisfying the above properties by Bertini theorem. Since not passing through the other points $x_{2},\dots,x_{r}$ is just an open condition, we may further assume that $x_{2},\dots,x_{r}$ are not contained in any zero loci of $t^{(1)}_{i}$, i.e. for each $i=0, \dots, n$, 
\begin{align*}
\nu_{Y^{1}_{\bullet},\dots, Y^{r}_{\bullet}}(t^{(1)}_{i})=({\bf e}_{i} \text{ ; } 0, \dots,0 \text{ ; } \dots \text{ ; } 0, \dots, 0).
\end{align*}
Similarly, we can construct very ample divisors $A_{1},\dots,A_{r},B_{1},\dots,B_{r}$ with sections $s^{(j)}_{0},t^{(j)}_{i}$ for each $j=1,\dots, r$ and $i=0,\dots,n$ such that 
\begin{align*}
\nu_{Y^{1}_{\bullet},\dots, Y^{r}_{\bullet}}(s^{(j)}_{0})=\nu_{Y^{1}_{\bullet},\dots, Y^{r}_{\bullet}}(t^{(j)}_{0})&=(0,\dots,0 \text{ ; } \dots \text{ ; } 0, \dots, 0) \in \mathbb{R}^{nr}, \\
\nu_{Y^{1}_{\bullet},\dots, Y^{r}_{\bullet}}(t^{(j)}_{i})&={\bf e}^{(j)}_{i} \in \mathbb{R}^{nr},
\end{align*}
where ${\bf e}^{(j)}_{i}$ is the $((j-1)n+i)$-th standard basis vector of $\mathbb{Z}^{nr}$. \

Since $D$ is big, there eixsts $m_{0} \in \mathbb{Z}$ such that $mD-B_{j}$ is linearly equivalent to an effective divisor $F_{m}^{(j)}$ for all $m \ge m_{0}$ and $j=1, \dots, r$, i.e. $mD=B_{j}+F^{(j)}_{m}$. If $f^{(j)}_{m} \in \mathbb{Z}^{nr}$ is the valuation vector of a section defining $F^{(j)}_{m}$, then the valuative property of $\nu_{Y^{1}_{\bullet},\dots, Y^{r}_{\bullet}}$ with the valuations of $t^{(j)}_{i}$ gives that 
\begin{align*}
(f^{(j)}_{m},m), (f^{(j)}_{m}+{\bf e}^{(j)}_{1},m), \dots, (f^{(j)}_{m}+{\bf e}^{(j)}_{n},m) \in \Gamma_{Y^{1}_{\bullet},\dots, Y^{r}_{\bullet}}(D)
\end{align*}
for each $j=1,\dots, r$. Moreover, $(m+1)D=A_{j}+F_{m}^{(j)}$ and so $\Gamma_{Y^{1}_{\bullet},\dots, Y^{r}_{\bullet}}(D)$ also contains $(f_{m}^{(j)}, m+1)$. Hence $\Gamma_{Y^{1}_{\bullet},\dots, Y^{r}_{\bullet}}(D) \subset \mathbb{N}^{nr+1}$ generates $\mathbb{Z}^{nr+1}$ as a group, that is, $\Gamma(X)$ generates $\mathbb{Z}^{nr}\times \mathbb{Z}^{\rho(X)}$ as a group, which proves the claim. 

Since ${\rm Supp}(\Gamma(X))$ (see \cite[Proposition 4.9]{LM} for a definition) is the closed convex cone spanned by $(a_{1}, \dots, a_{\rho(X)}) \in \mathbb{Z}^{\rho(X)}$ such that $H^{0}(X,\OO_{X}(a_{1}D_{1}+\dots+a_{\rho(X)}D_{\rho(X)})) \neq 0$, $(a_{1}, \dots, a_{\rho(X)})$ lies in the interior of ${\rm Supp}(\Gamma(X))$ if and only if $\OO_{X}(a_{1}D_{1}+\dots+a_{\rho(X)}D_{\rho(X)})$ is big. For such a vector ${\bf a}=(a_{1},\dots,a_{\rho(X)}) \in \mathbb{Z}^{\rho(X)}$, we can apply \cite[Proposition 4.9]{LM} so that
\begin{align*}
\Delta_{Y^{1}_{\bullet},\dots, Y^{r}_{\bullet}}(a_{1}D_{1}+\dots+a_{\rho(X)}D_{\rho(X)})&={\rm Cone}(\Gamma(X) \cap (\mathbb{N}^{nr} \times \mathbb{N} \cdot {\bf a})) \cap (\mathbb{R}^{nr} \times \{ {\bf a} \}) \\
&={\rm Cone}(\Gamma(X)) \cap (\mathbb{R}^{nr} \times \{ {\bf a} \}) \\
&=\Delta_{Y^{1}_{\bullet},\dots, Y^{r}_{\bullet}}(X) \cap (\mathbb{R}^{nr} \times \{ {\bf a} \}) \\
&=\Delta_{Y^{1}_{\bullet},\dots, Y^{r}_{\bullet}}(X) \cap {\rm pr}_{2}^{-1}(a_{1}D_{1}+\dots+a_{\rho(X)}D_{\rho(X)}),
\end{align*}
which verifies the theorem for any big integral divisors on $X$. The case of rational classes follows since both $\Delta_{Y^{1}_{\bullet},\dots, Y^{r}_{\bullet}}(D)$ and ${\rm pr}_{2}^{-1}(D) \cap \Delta_{Y^{1}_{\bullet},\dots, Y^{r}_{\bullet}}(X)$ scale linearly with $D$. 
\end{proof}

We refer to $\Delta_{Y^{1}_{\bullet},\dots, Y^{r}_{\bullet}}(X)$ in Theorem \ref{thm:global} as the global extended Okounkov body of $X$ with respect to $Y^{1}_{\bullet},\dots, Y^{r}_{\bullet}$.

\begin{rmk} \label{rmk:R-divisor}
For any $D, D' \in {{\rm Big}(X)}_{\mathbb{R}}$, 
\begin{enumerate}[(1)]
\item $\Delta_{Y^{1}_{\bullet},\dots, Y^{r}_{\bullet}}(D) \overset{\rm def}{=} \Delta_{Y^{1}_{\bullet},\dots, Y^{r}_{\bullet}}(X) \cap {\rm pr}_{2}^{-1}(D)$, and
\item $\Delta_{Y^{1}_{\bullet},\dots, Y^{r}_{\bullet}}(D)+\Delta_{Y^{1}_{\bullet},\dots, Y^{r}_{\bullet}}(D') \subseteq \Delta_{Y^{1}_{\bullet},\dots, Y^{r}_{\bullet}}(D+D')$.
\end{enumerate}
\end{rmk}

\begin{section}{Descriptions of extended Okounkov bodies} \label{section:4}
This section is devoted to some examples and computations. We start with curves. Let $L$ be a big $\mathbb{R}$-divisor on a smooth curve $C$, and let $Y^{i}_{\bullet}:C \supseteq \{p_{i}\}$ be admissible flags for $i=1,\dots,r$. Then it is easy to see that 
\begin{align*}
\Delta_{Y^{1}_{\bullet},\dots,Y^{r}_{\bullet}}(L)=\{a_{1}e_{1}+\cdots+a_{r}e_{r} \in \mathbb{R}^{r} \text{ }|\text{ } \sum_{i=1}^{r}a_{i} \le {\rm deg}L, \text{ } a_{i} \ge 0 \text{ for all } i=1,\dots, r \},
\end{align*}
where $e_{i}$ is the $i$-th standard basis vector. 

\begin{subsection}{Surfaces} \label{subsection:surfaces}
We observe a possible description on surface cases. Let $D$ be a big $\mathbb{R}$-divisor on a smooth projective surface $S$, and let $Y^{i}_{\bullet}:X \supseteq C_{i} \supseteq \{p_{i}\}$ be admissible flags, where $p_{i} \notin C_{j}$ for all $i \neq j$. For each $t_{1},\dots,t_{r} \in \mathbb{R}_{\ge 0}$, where $D-\sum_{i=1}^{r}t_{i}C_{i}$ is big, we write 
\begin{align*}
D-\sum_{i=1}^{r}t_{i}C_{i}=P_{t_{1},\dots,t_{r}}+N_{t_{1},\dots,t_{r}}
\end{align*}
for its Zariski decomposition. We set
\begin{align*}
\Delta=\{(\nu_{1}^{(1)},\nu_{2}^{(1)}; \cdots ; \nu_{1}^{(r)},\nu_{2}^{(r)}) \in \mathbb{R}^{2r} \text{ } | \text{ } &D-\sum_{i=1}^{r}(s_{i}+\nu_{1}^{(i)})C_{i} \text{ is big, } \alpha_{i}(\nu_{1}^{(1)},\dots,\nu_{1}^{(r)}) \le \nu_{2}^{(i)} \le \beta_{i}(\nu_{1}^{(1)},\dots,\nu_{1}^{(r)}) \\
&\text{ for each } i=1,\dots,r\},
\end{align*}
where $\alpha_{i}(\nu_{1}^{(1)},\dots,\nu_{1}^{(r)})={\rm ord}_{p_{i}}({N_{s_{1}+\nu_{1}^{(1)},\dots,s_{r}+\nu_{1}^{(r)}}|}_{C_{i}})$ and $\beta_{i}(\nu_{1}^{(1)},\dots,\nu_{1}^{(r)})=\alpha_{i}(\nu_{1}^{(1)},\dots,\nu_{1}^{(r)})+(P_{s_{1}+\nu_{1}^{(1)},\dots,s_{r}+\nu_{1}^{(r)}}.C_{i})$ for each $i=1,\dots,r$. Note that $\alpha_{i}(\nu_{1}^{(1)},\dots,\nu_{1}^{(r)})$ and $\beta_{i}(\nu_{1}^{(1)},\dots,\nu_{1}^{(r)})$ are piecewise linear (cf. \cite{BKS}). 

\begin{prop} \label{prop:surfaces}
With the above notation, the following inclusion holds:
\begin{align*}
\Delta_{Y^{1}_{\bullet},\dots,Y^{r}_{\bullet}}(D) \subseteq (s_{1},0; \cdots ;s_{r},0)+\overline{\Delta},
\end{align*}
where $s_{i}={\rm ord}_{C_{i}}(N_{0,\dots,0})$ and $\overline{\Delta}$ denotes its closure. 
\end{prop}

\begin{proof}
The proof of Lemma \ref{lem:slice} implies that
\begin{align*}
\Delta_{Y^{1}_{\bullet},\dots,Y^{r}_{\bullet}}(D) = (s_{1},0; \cdots ;s_{r},0)+\Delta_{Y^{1}_{\bullet},\dots,Y^{r}_{\bullet}}(D-s_{1}C_{1}-\dots -s_{r}C_{r})
\end{align*}
sicne the Zariski decomposition of $D$ is $P_{0,\dots,0}+N_{0,\dots,0}$ and $s_{i}={\rm ord}_{C_{i}}(N_{0,\dots,0})$ for all $i=1,\dots,r$. Note that any $C_{i}$'s do not appear in the negative part of $D-s_{1}C_{1}-\dots-s_{r}C_{r}$. \

It remains to show that $\Delta_{Y^{1}_{\bullet},\dots,Y^{r}_{\bullet}}(D-s_{1}C_{1}-\dots -s_{r}C_{r}) \subseteq \overline{\Delta}$. For a sufficiently large and divisible $m \gg 0$, let $\nu_{1}^{(i)}=\alpha^{(i)}$ for all $2 \le i \le r$ such that $mD-ms_{1}C_{1}-\sum_{i=2}^{r}(ms_{i}+m\alpha^{(i)})C_{i}$ is big and integral. Since all the sections in $H^{0}(X,\OO_{X}(mD))$, vanishing along $C_{1}$ with multiplicity $\ge ms_{1}$ and vanishing along $C_{i}$ with multiplicity exactly $ms_{i}+m\alpha^{(i)}$ for all $i \ge 2$, are contained in $H^{0}(X,\OO_{X}(mD-ms_{1}C_{1}-\sum_{i=2}^{r}(ms_{i}+m\alpha^{(i)})C_{i}))$ naturally, we have:
\begin{align*}
{\Delta_{Y^{1}_{\bullet},\dots,Y^{r}_{\bullet}}(D-s_{1}C_{1}-\sum_{i=2}^{r}(s_{i}+\nu_{1}^{(i)})C_{i})}_{\nu_{1}^{(i)}=\alpha^{(i)} \text{ for all } i \ge 2} \subseteq \Delta_{Y^{1}_{\bullet}}(D-s_{1}C_{1}-\sum_{i=2}^{r}(s_{i}+\nu_{1}^{(i)})C_{i})
\end{align*}
by Remark \ref{rmk:separate}. The above assertion with \cite[Theorem 6.4]{LM} implies that $\nu_{2}^{(1)}$ satisfies the inequality $\alpha_{1}(\nu_{1}^{(1)},\dots,\nu_{1}^{(r)}) \le \nu_{2}^{(1)} \le \beta_{1}(\nu_{1}^{(1)},\dots,\nu_{1}^{(r)})$. By repeating the same process for $\nu_{2}^{(i)}$ for other $i \neq 1$, we conclude that 
\begin{align*}
\Delta_{Y^{1}_{\bullet},\dots,Y^{r}_{\bullet}}(D-s_{1}C_{1}-\dots -s_{r}C_{r}) \subseteq \overline{\Delta}. 
\end{align*}
\end{proof}

\begin{rmk} \label{rmk:surfaces}
We expect that two convex bodies in Proposition \ref{prop:surfaces} are actually the same. 
\end{rmk}

\end{subsection}

\begin{subsection}{Toric varieties}
We describe the extended Okounkov bodies of toric varieties with respect to certain flags. The key point for its proof is the technique used in \cite[Subsection 6.1]{LM}. 

We start by fixing some notation. Let $X_{\Sigma}$ be a smooth projective toric variety induced by a fan $\Sigma \subseteq N_{\mathbb{R}} \cong \mathbb{R}^{n}$. Let $D=\sum_{\rho \text{ : ray}} a_{\rho}D_{\rho}$ be a torus invariant divisor on $X_{\Sigma}$ and write
\begin{align*}
P_{D}=\{ m \in M_{\mathbb{R}} \text{ } | \text{ } <m,u_{\rho}> \ge -a_{\rho} \text{ for all rays } \rho \in \Sigma \},
\end{align*}
where $M$ is the dual lattice of $N$ and $u_{\rho}$ is the ray generator of $\rho$. Note that $P_{mD}=mP_{D}$ for $m \in \mathbb{Z}_{>0}$ and $H^{0}(X_{\Sigma},\OO_{X_{\Sigma}}(D))=\bigoplus\limits_{m \in P_{D} \cap M} \mathbb{C} \cdot \chi^{m}$ by \cite[Proposition 4.3.3]{CLS}, where $\chi^{m}$ is the character map associated to $m$. 

Suppose that $\sigma_{1},\dots,\sigma_{r} \in \Sigma$ are maximal cones such that $\sigma_{1} \cap \cdots \cap \sigma_{r}=\{0\}$. (This is for the assumption on the flags needed to define the extended Okounkov bodies). For such $\sigma_{1},\dots,\sigma_{r}$, define the flags $Y^{i}_{\bullet}:X_{\Sigma}=Y^{i}_{0} \supseteq Y_{1}^{i} \supseteq \cdots \supseteq Y_{n}^{i}$ consisting of torus invariant subvarieties of $X_{\Sigma}$: For each $i=1,\dots,r$, we can order the prime torus invariant divisors $D_{1}^{(i)},\dots,D_{n}^{(i)}$ such that $Y_{j}^{i}=D_{1}^{(i)} \cap \cdots \cap D_{j}^{(i)}$. Moreover, for each ray generator $v_{j}^{(i)}$ of the ray corresponding to $D_{j}^{(i)}$, $\{v_{1}^{(i)},\dots,v_{n}^{(i)}\}$ form a basis of $N$ and they generate a cone $\sigma_{i}$ since the fan $\Sigma$ is smooth. 

Finally, we define a map $\bar{\phi}:M \rightarrow \mathbb{Z}^{nr}, u \mapsto ({(<u,v_{j}^{(1)}>)}_{1 \le j \le n}; \cdots ; {(<u,v_{j}^{(r)}>)}_{1 \le j \le n})$ and it determines an $\mathbb{R}$-linear map ${\bar{\phi}}_{\mathbb{R}}:M_{\mathbb{R}} \rightarrow \mathbb{R}^{nr}$ naturally. Note that $\bar{\phi}$ is injective. 

\begin{prop}
With the above notation, let $L$ be a big $\mathbb{R}$-divisor on $X_{\Sigma}$. Suppose that there exists a torus invariant divisor $D$ such that $L^{\otimes k} \cong \OO_{X_{\Sigma}}(D)$ for some $k$ and ${D|}_{U_{\sigma_{i}}}=0$ for all $i=1,\dots,r$, where $U_{\sigma_{i}}$ is the affine toric variety associated to $\sigma_{i}$. Then 
\begin{align*}
\Delta_{Y^{1}_{\bullet},\dots,Y^{r}_{\bullet}}(L)=\frac{1}{k} \cdot {\bar{\phi}}_{\mathbb{R}}(P_{D}).
\end{align*}
\end{prop}

\begin{proof}
Consider $u \in P_{D} \cap M$. Under the isomorphism $H^{0}(X_{\Sigma},\OO_{X_{\Sigma}}(D))=\bigoplus\limits_{m \in P_{D} \cap M} \mathbb{C} \cdot \chi^{m}$, the zero locus corresponding to $u$ is $D+\sum_{\rho \text{: ray}}<u,u_{\rho}>D_{\rho}$, where $u_{\rho}$ is the ray generator of $\rho$. Since ${D|}_{U_{\sigma_{i}}}=0$ for all $i=1,\dots,r$, and $\sigma_{1} \cap \sigma_{r}=\{0\}$, $\nu_{Y^{1}_{\bullet},\dots,Y^{r}_{\bullet}}(\chi^{u})=\bar{\phi}(u)$. Moreover, since $\phi$ is injective and there exists $h^{0}(L^{\otimes k})$ lattice points in $P_{D} \cap M$ (by the isomorphism $H^{0}(X_{\Sigma},L^{\otimes k})=\bigoplus\limits_{m \in P_{D} \cap M} \mathbb{C} \cdot \chi^{m}$), we have 
\begin{align*}
\bar{\phi}(P_{D} \cap M)=\text{Image of } ((H^{0}(X_{\Sigma},L^{\otimes k})-\{0\}) \overset{\small{\nu_{Y^{1}_{\bullet},\dots, Y^{r}_{\bullet}}}}{\longrightarrow} \mathbb{Z}^{nr}).
\end{align*}
Note that for a sufficiently large and divisible $m \gg 0$, $mP_{D}$ has all its vertices in $M$ so that $(\text{the convex hull of } \frac{1}{m} \bar{\phi}(mP_{D} \cap M))={\bar{\phi}}_{\mathbb{R}}(P_{D})$. By the homogeneity of extended Okounkov bodies and the fact that $P_{mD}=mP_{D}$, we conclude that $\Delta_{Y^{1}_{\bullet},\dots,Y^{r}_{\bullet}}(L)=\frac{1}{k} \cdot {\bar{\phi}}_{\mathbb{R}}(P_{D})$. 
\end{proof}

\end{subsection}

\end{section}

\end{section}

\begin{section}{Asymptotic base loci and multi-weight moving Seshadri constant via extended infinitesimal Okounkov bodies} \label{section:5}

This section is the core of the paper. We define the extended infinitesimal Okounkov bodies and discuss the asymptotic base loci and the multi-weight moving Seshadri constants in terms of extended infinitesimal Okounkov bodies. First, we recall the definition of infinitesimal flags. 

\begin{defn} (\cite[Definition 2.1]{KL1507}) 
Let $\phi:{\rm Bl}_{x}(X) \rightarrow X$ be the blow-up of $X$ at a point $x$ with exceptional divisor $E$. An infinitesimal flag $Y_{\bullet}$ over $x$ is an admissible flag $Y_{\bullet}:Y_{0}={\rm Bl}_{x}(X) \supset Y_{1}=E \supset \dots \supset Y_{n}$, where each $Y_{i}$ is a linear subspace of $E \cong \mathbb{P}^{n-1}$ of dimension $n-i$ for each $i=2, \dots, n$. 
\end{defn}

\begin{defn} 
Let $\pi:\tilde{X}={\rm Bl}_{\{x_{1}, \dots, x_{r}\}}(X) \rightarrow X$ be the blow-up of $X$ at $x_{1}, \dots, x_{r} \in X$, and let $Y^{i}_{\bullet}$ be an infinitesimal flag over $x_{i}$ for each $i=1, \dots, r$. The extended infinitesimal Okounkov body of a big $\mathbb{R}$-divisor $D$ with respect to $Y^{1}_{\bullet}, \dots, Y^{r}_{\bullet}$ is the extended Okounkov body $\Delta_{Y^{1}_{\bullet}, \dots, Y^{r}_{\bullet}}(\pi^{*}D)$. We simply write it $\widetilde{\Delta}_{Y^{1}_{\bullet}, \dots, Y^{r}_{\bullet}}(D)$.
\end{defn}

The inverted standard slice simplex $\Delta_{(\xi_{1},\dots,\xi_{r})}^{-r}$ of size $(\xi_{1},\dots,\xi_{r}) \in \mathbb{R}_{\ge 0}^{r}$ in $\mathbb{R}^{nr}$ is the convex hull of the set
\begin{align*}
\{{\bf 0}, \text{ } {\bf v}^{(\xi_{1},\dots,\xi_{r})}_{1}, {\bf v}^{(\xi_{1},\dots,\xi_{r})}_{1}+{\bf v}^{(\xi_{1},\dots,\xi_{r})}_{2}, \dots, \text{ } {\bf v}^{(\xi_{1},\dots,\xi_{r})}_{1}+{\bf v}^{(\xi_{1},\dots,\xi_{r})}_{n}\} \subseteq \mathbb{R}^{nr}.
\end{align*}
In particular, when $\xi_{1}=\dots=\xi_{r}=\xi$, then $\Delta_{(\xi,\dots,\xi)}^{-r}$ will be denoted by $\Delta_{\xi}^{-r}$. 

\begin{rmk}
$\Delta_{(\xi_{1},\dots,\xi_{r})}^{-r}$ can be viewed as an inverted standard simplex of size $\sqrt{\sum_{i=1}^{r} \xi_{i}^{2}}$ in $S_{\xi_{1},\dots,\xi_{r}} \cong \mathbb{R}^{n}$. 
\end{rmk}

A major difference between the infinitesimal Okounkov bodies and the non-infinitesimal ones is that the infinitesimal Okounkov bodies are contained in some inverted standard simplex in a very natural way (\cite[Proposition 2.6]{KL1507}). We can say the similar thing in the case of extended Okounkov bodies. 

\begin{lem} \label{lem:upper bound}
Let $D$ be a big $\mathbb{R}$-divisor on $X$ with $x_{1},\dots,x_{r} \in X$. Then there exists a constant $\mu(D;x_{1},\dots,x_{r})>0$ such that 
\begin{align*}
\Delta_{Y^{1}_{\bullet}, \dots, Y^{r}_{\bullet}}(\pi^{*}D) \subseteq r \cdot {\rm Conv}(\bigcup_{i=1}^{r} {\rm pr}_{i}(\Delta^{-r}_{\mu(D;x_{1},\dots,x_{r})}))
\end{align*}
for any infinitesimal flags $Y^{1}_{\bullet},\dots,Y^{r}_{\bullet}$ over $x_{1},\dots,x_{r}$, where ${\rm pr}_{i}$ is a projection map in Note \ref{note:projection}. 
\end{lem} 

\begin{proof}
Let $\mu(D;x_{i})=\sup \{s>0 \text{ }|\text{ $\pi^{*}D-sE_{i}$ is big}\}$, and set $\mu(D;x_{1},\dots,x_{r})=\max_{i=1,\dots, r} \{\mu(D;x_{i})\}$. Then the result follows from \cite[Proposition 2.6]{KL1507} and Note \ref{note:projection}. 
\end{proof}

\begin{subsection}{Restricted base loci via extended Okounkov bodies}

One defines the restricted base locus of a big $\mathbb{R}$-divisor $D$ as
\begin{align*}
{\rm \bf{B}}_{-}(D) \overset{\rm def}{=} \bigcup_{A} {\rm \bf{B}}(D+A),
\end{align*}
where the union is taken over all ample divisors $A$, such that $D+A$ is a $\mathbb{Q}$-divisor (\cite[Definition 1.12]{ELMNP2}). This locus is a countable union of closed subvarieties of $X$ (\cite[Proposition 1.19]{ELMNP2}). 

\begin{prop} \label{prop:restricted base loci}
Let $D$ be a big $\mathbb{R}$-divisor on $X$. Then the following are equivalent. 
\begin{enumerate}[(1)]
\item $x_{1},\dots,x_{r} \notin {\rm \bf{B}}_{-}(D)$.
\item There exist infinitesimal flags $Y^{1}_{\bullet},\dots, Y^{r}_{\bullet}$ over $x_{1},\dots,x_{r}$ such that ${\bf 0} \in \widetilde{\Delta}_{Y^{1}_{\bullet}, \dots,Y^{r}_{\bullet}}(D)$. 
\item For every infinitesimal flags $Y^{1}_{\bullet},\dots, Y^{r}_{\bullet}$ over $x_{1},\dots,x_{r}$, one has ${\bf 0} \in \widetilde{\Delta}_{Y^{1}_{\bullet}, \dots,Y^{r}_{\bullet}}(D)$.
\end{enumerate}
\end{prop}

\begin{proof}
$(1) \Rightarrow (3)$ Suppose not, i.e. there exist infinitesimal flags $Y^{i}_{\bullet}$ over $x_{i}$ such that ${\bf 0} \notin \widetilde{\Delta}_{Y^{1}_{\bullet}, \dots,Y^{r}_{\bullet}}(D)$. Let ${\rm pr}_{i}:\mathbb{R}^{nr} \rightarrow \mathbb{R}^{n}$ be a projection as in Note \ref{note:projection}. Then there exists $i$ such that $(0,\dots,0) \notin {\rm pr}_{i}(\widetilde{\Delta}_{Y^{1}_{\bullet}, \dots,Y^{r}_{\bullet}}(D))$, that is, $(0, \dots, 0) \notin \widetilde{\Delta}_{Y^{i}_{\bullet}}(D)$. However, it contradicts \cite[Theorem 3.1]{KL1507}. 

The implication $(3) \Rightarrow (2)$ is obvious, so we only need to check $(2) \Rightarrow (1)$. It is a direct consequence of Note \ref{note:projection} and \cite[Theorem 3.1]{KL1507}, so we are done.  
\end{proof}

\end{subsection}

\begin{subsection} {Moving multi-weight Seshadri constants and augmented base loci via extended Okounkov bodies}

We recall the necessary information about multi-weight moving Seshadri constants and connect them with the extended Okounkov bodies. 

The augmented base locus of a big $\mathbb{R}$-divisor $D$ is defined by 
\begin{align*}
{\rm \bf{B}}_{+}(D) \overset{\rm def}{=} \bigcap_{A} {\rm \bf{B}}(D-A),
\end{align*}
where the intersection is taken over all ample divisors $A$, such that $D-A$ is a $\mathbb{Q}$-divisor (\cite[Definition 1.12]{ELMNP2}). 

\begin{defn} (\cite{HR, Na, AT}) \label{defn:multi-weight}
\begin{enumerate}[(1)]
\item The multi-weight Seshadri constant of an ample $\mathbb{R}$-divisor $A$ on $X$ with weight ${\bf m}=(m_{1},\dots,m_{r}) \in \mathbb{N}^{r}$ at $x_{1},\dots,x_{r} \in X$ is the real number
\begin{align*}
\epsilon_{\bf m}(A;x_{1},\dots,x_{r}) \overset{\rm def}{=} \sup \{a \in \mathbb{R} \text{ }|\text{ } \pi^{*}A-a\cdot \sum_{i=1}^{r}m_{i}E_{i} \text{ is ample on } \tilde{X}\}.
\end{align*}
\item The multi-weight moving Seshadri constant of a big $\mathbb{R}$-divisor $D$ on $X$ with weight ${\bf m}=(m_{1},\dots,m_{r}) \in \mathbb{N}^{r}$ at $x_{1},\dots,x_{r} \in X$ is the real number
\begin{align*}
\epsilon_{\bf m}(\lVert D \rVert;x_{1},\dots,x_{r})&=\sup_{f^{*}L=A+E} \epsilon_{\bf m}(A;f^{-1}(x_{1}),\dots,f^{-1}(x_{r})) \text{ if } x_{1},\dots,x_{r} \notin {\rm \bf{B}}_{+}(D) \\
& =0 \text{ otherwise}
\end{align*}
where the supremum is taken over all projective resolutions $f:Y \rightarrow X$ with $f$ an isomorphism around $x_{1},\dots,x_{r}$ and over all decompositions $f^{*}D=A+E$, where $A$ is an ample $\mathbb{Q}$-divisor and $E$ is effective with $f^{-1}(x_{i}) \notin {\rm Supp}(E)$ for all $i$. 

When ${\bf m}=(1,\dots,1)$, we write it simply $\epsilon(\lVert D \rVert;x_{1},\dots,x_{r})$. 
\end{enumerate}
\end{defn}

\begin{lem} \label{lem:inequality}
Let $D$ be a big $\mathbb{R}$-divisor on $X$, ${\bf m}=(m_{1},\dots,m_{r}) \in \mathbb{N}^{r}$, and $x_{1}, \dots,x_{r} \notin {\rm \bf{B}}_{-}(D)$. For any infinitesimal flags $Y^{1}_{\bullet}, \dots, Y^{r}_{\bullet}$ over $x_{1}, \dots,x_{r}$, 
\begin{align*}
\Delta^{-r}_{(\eta_{1}, \dots, \eta_{r})} \subseteq \widetilde{\Delta}_{Y^{1}_{\bullet}, \dots, Y^{r}_{\bullet}}(D), 
\end{align*}
where $\eta_{i}=m_{i} \cdot \epsilon_{\bf m}(\lVert D \rVert;x_{1},\dots,x_{r})$. 
\end{lem}

\begin{proof}
Once we show it for big line bundles, the general one follows easily by the continuity of the extended Okounkov bodies. So we only deal with a big line bundle $D$ on $X$. First, we consider the mono-graded case, i.e. ${\bf m}=(1,\dots,1)$. 

Assume that $D$ is ample. Choose and fix a rational number $0 < \epsilon < \epsilon(D;x_{1},\dots,x_{r})$ and an infinitesimal flag $Y^{i}_{\bullet}$ over $x_{i}$ for each $i=1,\dots,r$. It is enough to show that $\Delta^{-r}_{\epsilon} \subseteq \widetilde{\Delta}_{Y^{1}_{\bullet}, \dots, Y^{r}_{\bullet}}(D)$. We write 
\begin{align*}
K_{\epsilon}=\pi^{*}D-\epsilon \cdot \sum_{i}^{r} E_{i}. 
\end{align*}
Note that $E_{i} \cong \mathbb{P}^{n-1}$ for all $i=1, \dots, r$, and $E_{i}$'s are all disjoint. Moreover, 
\begin{align*}
{\rm deg}_{E_{i}}({(mK_{\epsilon})|}_{E_{i}})={\rm deg}_{E_{i}}({(m\pi^{*}D-m\epsilon E_{i})|}_{E_{i}})=({(m\pi^{*}D-m\epsilon E_{i})|}_{E_{i}}.\OO_{E_{i}}(1)^{n-2})=m\epsilon
\end{align*}
for all sufficiently divisible $m>0$. \

For all sufficiently divisible $m \gg 0$, a short exact sequence 
\begin{align*}
0 \rightarrow \OO_{\tilde{X}}(mK_{\epsilon}-E_{1}-\dots-E_{r}) \rightarrow \OO_{\tilde{X}}(mK_{\epsilon}) \rightarrow \OO_{E_{1} \sqcup \dots \sqcup E_{r}}({mK_{\epsilon}|}_{E_{1} \sqcup \dots \sqcup E_{r}}) \rightarrow 0
\end{align*}
induces an exact sequence
{\small
\begin{align*}
0 \rightarrow H^{0}(\tilde{X}, \OO_{\tilde{X}}(mK_{\epsilon}-E_{1}-\dots-E_{r})) \rightarrow H^{0}(\tilde{X}, \OO_{\tilde{X}}(mK_{\epsilon})) \overset{\nu}{\rightarrow} H^{0}(E_{1} \sqcup \dots \sqcup E_{r}, \OO_{E_{1} \sqcup \dots \sqcup E_{r}}({mK_{\epsilon}|}_{E_{1} \sqcup \dots \sqcup E_{r}})) \rightarrow 0
\end{align*}
}
since $K_{\epsilon}$ is ample and $m \gg 0$. Note that 
\begin{align*}
H^{0}(E_{1} \sqcup \dots \sqcup E_{r}, \OO_{E_{1} \sqcup \dots \sqcup E_{r}}({mK_{\epsilon}|}_{E_{1} \sqcup \dots \sqcup E_{r}})) \cong H^{0}(E_{1}, \OO_{E_{1}}(m\epsilon)) \oplus \cdots \oplus H^{0}(E_{r}, \OO_{E_{r}}(m\epsilon)).
\end{align*}
Consider a non-zero section $(s^{(1)}_{1}, \dots, s^{(1)}_{r})$ whose $s^{(1)}_{i} \in H^{0}(E_{i}, \OO_{E_{i}}({mK_{\epsilon}|}_{E_{i}}))=H^{0}(E_{i}, \OO_{E_{i}}(m\epsilon))$ is a section such that $\nu_{{E_{i}}_{\bullet}}(s_{i}^{(1)})=(m\epsilon, 0, \dots, 0)$ for each $i=1, \dots, r$, where $\nu_{{E_{i}}_{\bullet}}$ is the valuation-like function attached to ${E_{i}}_{\bullet}$ (\cite[Subsection 1.1]{LM}). By the surjectivity of $\nu$, $(s_{1}^{(1)},\dots,s_{r}^{(1)})$ can be lifted to $\bar{s}^{(1)} \in H^{0}(\tilde{X},\OO_{\tilde{X}}(mK_{\epsilon}))$ and it induces a section $s^{(1)} \in H^{0}(\tilde{X}, \OO_{\tilde{X}}(m\pi^{*}D))$ that satisfies 
\begin{align*}
\nu_{Y^{1}_{\bullet}, \dots, Y^{r}_{\bullet}}(s^{(1)})=(m\epsilon, m\epsilon, 0, \dots, 0 \text{ ; } \dots \text{ ; } m\epsilon, m\epsilon, 0, \dots, 0). 
\end{align*}

Similarly, we can construct $s^{(0)}, \dots, s^{(r)} \in H^{0}(\tilde{X}, \OO_{\tilde{X}}(m\pi^{*}D))$ such that 
\begin{align*}
\nu_{Y^{1}_{\bullet}, \dots, Y^{r}_{\bullet}}(s^{(0)})&=(m\epsilon, 0, \dots, 0 \text{ ; } \dots \text{ ; } m\epsilon, 0, \dots, 0), \\
\nu_{Y^{1}_{\bullet}, \dots, Y^{r}_{\bullet}}(s^{(1)})&=(m\epsilon, m\epsilon, 0, \dots, 0 \text{ ; } \dots \text{ ; } m\epsilon, m\epsilon, 0, \dots, 0), \\
& \vdots \\
\nu_{Y^{1}_{\bullet}, \dots, Y^{r}_{\bullet}}(s^{(r)})&=(m\epsilon, 0, \dots, 0, m\epsilon \text{ ; } \dots \text{ ; } m\epsilon, 0, \dots, 0, m\epsilon). 
\end{align*}
It implies that ${\bf v}^{(\epsilon,\dots,\epsilon)}_{1}, {\bf v}^{(\epsilon,\dots,\epsilon)}_{1}+{\bf v}^{(\epsilon,\dots,\epsilon)}_{2}, \dots, {\bf v}^{(\epsilon,\dots,\epsilon)}_{1}+{\bf v}^{(\epsilon,\dots,\epsilon)}_{n}$ are contained in $\widetilde{\Delta}_{Y^{1}_{\bullet}, \dots, Y^{r}_{\bullet}}(D)$. Since ${\bf 0} \in \widetilde{\Delta}_{Y^{1}_{\bullet}, \dots, Y^{r}_{\bullet}}(D)$, we have $\Delta^{-r}_{\epsilon(D;x_{1},\dots,x_{r})} \subseteq \widetilde{\Delta}_{Y^{1}_{\bullet}, \dots, Y^{r}_{\bullet}}(D)$. 

Next, assume that $D$ is big. If one of the $x_{i}$'s is contained in ${\rm \bf{B}}_{+}(D)$, then $\epsilon(\lVert D \rVert;x_{1},\dots,x_{r})=0$. However, since $x_{1},\dots,x_{r} \notin {\rm \bf{B}}_{-}(D)$, $\widetilde{\Delta}_{Y^{1}_{\bullet}, \dots, Y^{r}_{\bullet}}(D)$ contains the origin (Proposition \ref{prop:restricted base loci}), so we are done in this case. So let $x_{1},\dots,x_{r} \notin {\rm \bf{B}}_{+}(D)$. Then we can find a projective resolution $f$ and a decomposition $f^{*}D=A+E$ as in Definition \ref{defn:multi-weight}. Fix any such $f$, $A$, and $E$. Then we have the commutative diagram: 
\begin{displaymath}
\xymatrix{
{\rm Bl}_{\{f^{-1}(x_{1}),\dots,f^{-1}(x_{r})\}}(Y) \ar[d]^{\pi_{Y}} \ar[rr]^{\phi} && {\rm Bl}_{\{x_{1},\dots,x_{r}\}}(X)=\tilde{X} \ar[d]^{\pi} \\
Y \ar[rr]_{f} && X.}
\end{displaymath}
Since $f$ is an isomorphism around $x_{1},\dots,x_{r}$, there are induced infinitesimal flags $\tilde{Y}^{1}_{\bullet},\dots, \tilde{Y}^{r}_{\bullet}$ over $f^{-1}(x_{1}),\dots,f^{-1}(x_{r})$ by the strict transforms of $Y^{1}_{\bullet},\dots,Y^{r}_{\bullet}$. It is easy to see that 
\begin{align*}
\widetilde{\Delta}_{Y^{1}_{\bullet}, \dots, Y^{r}_{\bullet}}(D)=\widetilde{\Delta}_{\tilde{Y}^{1}_{\bullet}, \dots, \tilde{Y}^{r}_{\bullet}}(f^{*}D). 
\end{align*}
Choose a non-zero section $s$ whose support is $E$. Then, for any non-zero section $t_{m} \in H^{0}(Y,\OO_{Y}(mA))$, $t_{m} \otimes s^{\otimes m} \in H^{0}(Y,f^{*}\OO_{Y}(mD))$. Since $f^{-1}(x_{1}),\dots,f^{-1}(x_{r}) \notin E$, 
\begin{align*}
\widetilde{\Delta}_{\tilde{Y}^{1}_{\bullet}, \dots, \tilde{Y}^{r}_{\bullet}}(A) \subseteq \widetilde{\Delta}_{\tilde{Y}^{1}_{\bullet}, \dots, \tilde{Y}^{r}_{\bullet}}(f^{*}D). 
\end{align*}
Since we proved the conclusion for ample line bundles in the mono-graded case, we have
\begin{align*}
\Delta^{-r}_{\epsilon(A;f^{-1}(x_{1}),\dots,f^{-1}(x_{r}))} \subseteq \widetilde{\Delta}_{\tilde{Y}^{1}_{\bullet}, \dots, \tilde{Y}^{r}_{\bullet}}(A) \subseteq \widetilde{\Delta}_{\tilde{Y}^{1}_{\bullet}, \dots, \tilde{Y}^{r}_{\bullet}}(f^{*}D)=\widetilde{\Delta}_{Y^{1}_{\bullet}, \dots, Y^{r}_{\bullet}}(D). 
\end{align*}
Since we choose arbitrary $f$, $A$, and $E$ in Definition \ref{defn:multi-weight}, $\Delta^{-r}_{\epsilon(\lVert D \rVert;x_{1},\dots,x_{r})} \subseteq \widetilde{\Delta}_{Y^{1}_{\bullet}, \dots, Y^{r}_{\bullet}}(D)$, which proves in the mono-graded case. 

For a general ${\bf m} \in \mathbb{N}^{r}$, by changing $K_{\epsilon}=\pi^{*}D-\epsilon \cdot \sum_{i=1}^{r} E_{i}$ (respectively ${\rm deg}_{E_{i}}({(mK_{\epsilon})|}_{E_{i}})=m\epsilon$) to $K_{\epsilon}=\pi^{*}D-\epsilon \cdot \sum_{i=1}^{r}m_{i}E_{i}$ (respectively ${\rm deg}_{E_{i}}({(mK_{\epsilon})|}_{E_{i}})=mm_{i}\epsilon$) and using the same argument, we can draw the conclusion. 
\end{proof}

We recall the relation between multi-weight moving Seshadri constants and jet separations. Main references are \cite[Chapter 5]{L}, \cite{I}, \cite{ELMNP1}, and \cite{AT}.  

\begin{defn} \label{defn:multi jet separation}
Let $D$ be a line bundle on $X$. 
\begin{enumerate}[(i)]
\item We say that $D$ separates ${\bf m}$-jets at $x_{1},\dots,x_{r}$ if the natural map 
\begin{align*}
H^{0}(X,\OO_{X}(D)) \rightarrow H^{0}(X,\OO_{X}(D)\otimes (\oplus_{i=1}^{r}\OO_{X}/m_{x_{i}}^{m_{i}+1}))
\end{align*}
is surjective, where $m_{x_{i}}$ is the ideal sheaf of $x_{i}$. 
\item Denote by $s_{\bf m}(D;x_{1},\dots,x_{r})$ the largest real number $s$ such that $D$ separates $\lceil s \cdot {\bf m} \rceil$-jets at $x_{1},\dots,x_{r}$, where $\lceil s \cdot {\bf m} \rceil:=(\lceil s \cdot m_{1} \rceil,\dots, \lceil s \cdot m_{r} \rceil)$.
\end{enumerate}
\end{defn}

\begin{lem} (\cite{I, AT}) \label{lem:continuity}
Let $D$ be a big $\mathbb{Q}$-divisor on $X$. 
\begin{enumerate}[(1)]
\item $\sup_{k \in \mathbb{N}} \frac{s_{\bf m}(kD;x_{1},\dots,x_{r})}{k}=\limsup_{k \rightarrow \infty} \frac{s_{\bf m}(kD;x_{1},\dots,x_{r})}{k}=\lim_{k \rightarrow \infty} \frac{s_{\bf m}(kD;x_{1},\dots,x_{r})}{k}$.
\item $\epsilon_{\bf m}(\lVert D \rVert;x_{1},\dots,x_{r})=\lim_{k \rightarrow \infty} \frac{s_{\bf m}(kD;x_{1},\dots,x_{r})}{k}$. 
\item The multi-weight moving Seshadri constant $\epsilon(\lVert - \rVert; x_{1},\dots,x_{r})$ is continuous on ${\rm Big}(X)_{\mathbb{R}}$. 
\end{enumerate}
\end{lem}

In order to study the multi-weight moving Seshadri constants via extended infinitesimal Okounkov bodies, we follow the steps of the proof of \cite[Proposition 4.10]{KL1507}. For them, we mimic the proofs of \cite[Theorem 4.24]{LM} and \cite[Proposition 4.9]{KL1507} with some modifications. 

\begin{lem} \label{lem:slice}
Let $D$ be a big $\mathbb{R}$-divisor on $X$. For any infinitesimal flags $Y^{1}_{\bullet}, \dots, Y^{r}_{\bullet}$ over $x_{1}, \dots,x_{r}$, and $a, m_{1}, \dots, m_{r} \ge 0$ such that $\pi^{*}D-a \cdot \sum_{i=1}^{r}m_{i}E_{i}$ is big, 
\begin{align*}
{\Delta_{Y^{1}_{\bullet}, \dots, Y^{r}_{\bullet}}(D)}_{\nu_{1}^{(1)} \ge am_{1}, \dots, \nu_{1}^{(r)} \ge am_{r}}=\Delta_{Y^{1}_{\bullet}, \dots, Y^{r}_{\bullet}}(\pi^{*}D-a \cdot \sum_{i=1}^{r}m_{i}E_{i})+(am_{1},0,\dots,0 \text{ ; } \dots \text{ ; }am_{r},0,\dots,0).
\end{align*}
\end{lem}

\begin{proof}
By the continuity of the extended Okounkov bodies (Theorem \ref{thm:global}), we may assume that $D$ is a big $\mathbb{Q}$-divisor. Multiplying some scalars to both sides if necessary, we may assume that $a \in \mathbb{N}$ and that $D$ is integral. Note that 
\begin{align*}
H^{0}(\tilde{X},\OO_{\tilde{X}}(m\pi^{*}D-ma \cdot \sum_{i=1}^{r}m_{i}E_{i}))&=\{s \in H^{0}(\tilde{X},\OO_{\tilde{X}}(m\pi^{*}D)) \text{ }|\text{ } {\rm ord}_{E_{i}}(s) \ge mam_{i} \text{ for all $i$}\} \\
&=\{s \in H^{0}(\tilde{X},\OO_{\tilde{X}}(m\pi^{*}D)) \text{ }|\text{ } \nu_{1}^{(i)}(s) \ge mam_{i} \text{ for all $i$}\}.
\end{align*}
In view of $\nu_{Y^{1}_{\bullet}, \dots, Y^{r}_{\bullet}}$, 
\begin{align*}
{\Gamma_{Y^{1}_{\bullet}, \dots, Y^{r}_{\bullet}}(\pi^{*}D)}_{\nu_{1}^{(1)} \ge am_{1}, \dots, \nu_{1}^{(r)} \ge am_{r}}=\phi_{(am_{1},\dots,am_{r})}(\Gamma_{Y^{1}_{\bullet}, \dots, Y^{r}_{\bullet}}(\pi^{*}D-a \cdot \sum m_{i}E_{i})), 
\end{align*}
where $\phi_{(am_{1},\dots,am_{r})}:\mathbb{N}^{nr}\times \mathbb{N} \rightarrow \mathbb{N}^{nr}\times \mathbb{N}$, $(s, m) \mapsto (s+mam_{1}\nu_{1}^{(1)}+\dots+mam_{r}\nu_{1}^{(r)}, m)$. \

Passing to cones, 
\begin{align*}
{\rm Cone}({\Gamma_{Y^{1}_{\bullet}, \dots, Y^{r}_{\bullet}}(\pi^{*}D)}_{\nu_{1}^{(1)} \ge am_{1}, \dots, \nu_{1}^{(r)} \ge am_{r}})=\phi_{(am_{1},\dots,am_{r}),\mathbb{R}}({\rm Cone}(\Gamma_{Y^{1}_{\bullet}, \dots, Y^{r}_{\bullet}}(\pi^{*}D-a \cdot \sum m_{i}E_{i})),
\end{align*}
where $\phi_{(am_{1},\dots,am_{r}),\mathbb{R}}:\mathbb{R}^{nr}\times \mathbb{R} \rightarrow \mathbb{R}^{nr}\times \mathbb{R}$ is a natural map induced by $\phi_{(am_{1},\dots,am_{r})}$. Cutting them by $\mathbb{R}^{nr} \times \{1\}$, the conclusion holds.
\end{proof}


\begin{prop} \label{prop:jet separation}
Let $D$ be a big line bundle on $X$. Assume that there exists $\epsilon \in \mathbb{R}_{>0}$ and $k \in \mathbb{N}$ such that $\Delta_{(n+km_{1}+\epsilon,\dots,n+km_{r}+\epsilon)}^{-r} \subseteq \widetilde{\Delta}_{Y^{1}_{\bullet}, \dots, Y^{r}_{\bullet}}(D)$ for all infinitesimal flags $Y^{1}_{\bullet}, \dots, Y^{r}_{\bullet}$ over $x_{1}, \dots, x_{r}$. Then $K_{X}+D$ separates $k\cdot {\bf m}$-jets at $x_{1},\dots, x_{r}$, where ${\bf m}=(m_{1},\dots,m_{r})$. 
\end{prop}

\begin{proof}
First, we deal with the mono-graded case, i.e. $m_{1}=\dots=m_{r}=1$. Projection formula implies that it suffices to show that 
\begin{align*}
H^{0}(\tilde{X}, \OO_{\tilde{X}}(\pi^{*}(K_{X}+D))) \rightarrow H^{0}(\tilde{X}, \OO_{\tilde{X}}(\pi^{*}(K_{X}+D)) \otimes (\oplus_{i=1}^{r}\OO_{\tilde{X}}/\OO_{\tilde{X}}(-(k+1)E_{i})))
\end{align*}
is surjective. Let $B:=\pi^{*}D-\sum_{i=1}^{r}(n+k)E_{i}$. Since $\Delta_{n+k+\epsilon}^{-r} \subseteq {\widetilde{\Delta}_{Y^{1}_{\bullet}, \dots, Y^{r}_{\bullet}}(D)}$, $B$ is big on $\tilde{X}$. So by Lemma \ref{lem:slice}, 
\begin{align*}
{\Delta_{Y^{1}_{\bullet}, \dots, Y^{r}_{\bullet}}(B)}={\widetilde{\Delta}_{Y^{1}_{\bullet}, \dots, Y^{r}_{\bullet}}(D)}_{\nu_{1}^{(1)}, \dots, \nu_{1}^{(r)} \ge n+k}-(n+k,0,\dots,0 \text{ ; } \dots \text{ ; }n+k,0,\dots,0).
\end{align*}
Moreover, since ${\Delta_{Y^{1}_{\bullet}, \dots, Y^{r}_{\bullet}}(B)}$ contains the origin, so is $\Delta_{Y^{i}_{\bullet}}(B)$ for each $i=1, \dots, r$. Since it holds for any infinitesimal flags $Y^{i}_{\bullet}$, any point $z$ in $E_{i}$ is not contained in ${\rm \bf{B}}_{-}(B)$ by \cite[Theorem 2.1]{KL1506}, that is, ${\rm \bf{B}}_{-}(B) \cap E_{i}=\emptyset$ for all $i=1, \dots, r$. Then \cite[Corollary 2.10]{ELMNP2} gives that ${\rm Zeroes}(\mathcal{J}(X,||B||)) \cap E_{i} =\emptyset$ for all $i=1, \dots, r$. \

Note that $K_{\tilde{X}}=\pi^{*}K_{X}+\sum_{i=1}^{r}(n-1)E_{i}$. Thus we have the short exact sequence
\begin{align*}
0 \rightarrow &\OO_{\tilde{X}}(K_{\tilde{X}}+B) \otimes \mathcal{J}(\tilde{X},||B||) \rightarrow \OO_{\tilde{X}}(\pi^{*}(K_{X}+D)) \rightarrow \mathcal{F} \rightarrow 0, 
\end{align*}
where $\mathcal{F}=\OO_{\tilde{X}}(\pi^{*}(K_{X}+D)) \otimes (\OO_{{\rm Zeroes}(\mathcal{J}(\tilde{X},||B||))} \oplus \OO_{(k+1)E_{1}} \oplus \dots \oplus \OO_{(k+1)E_{r}})$ since ${\rm Zeroes}(\mathcal{J}(\tilde{X},||B||))$, $E_{1}, \dots, E_{r}$ are all mutually disjoint. Since $B$ is big, by \cite[Theorem 11.2.12-(ii)]{L}, 
\begin{align*}
H^{1}(\tilde{X},\OO_{\tilde{X}}(K_{\tilde{X}}+B)\otimes \mathcal{J}(X,||B||))=0,
\end{align*}
that is, 
\begin{align*}
H^{0}(\tilde{X}, \OO_{\tilde{X}}(\pi^{*}(K_{X}+D))) \rightarrow H^{0}(\tilde{X}, \mathcal{F}) 
\end{align*}
is surjective. Hence the desired surjectivity holds. 

For a general ${\bf m} \in \mathbb{N}^{r}$, the proof is the same except $B=\pi^{*}D-\sum_{i=1}^{r}(n+km_{i})E_{i}$.
\end{proof}

We are in a position to connect the multi-weight moving Seshadri constants with the extended Okounkov bodies. 

\begin{defn} \label{defn:largest inverted standard simplex constant}
The largest inverted slice simplex constant $\xi_{\bf m}(D;x_{1},\dots,x_{r})$ of a big $\mathbb{R}$-divisor $D$ at $x_{1},\dots,x_{r}$ with multi-weight ${\bf m}=(m_{1},\dots,m_{r}) \in \mathbb{N}^{r}$ is the real number
\begin{align*}
\xi_{\bf m}(D;x_{1},\dots,x_{r}) \overset{\rm def}{=} \sup_{Y^{i}_{\bullet}} \{a \ge 0 \text{ } | &\text{ } \Delta_{(m_{1}a,\dots,m_{r}a)}^{-r} \subseteq \widetilde{\Delta}_{Y^{1}_{\bullet}, \dots, Y^{r}_{\bullet}}(D)\},
\end{align*}
where the supremum is taken over all infinitesimal flags $Y^{i}_{\bullet}$ over $x_{i}$. When $x_{i} \in {\rm \bf B}_{-}(D)$ for some $i$, we let $\xi_{\bf m}(D;x_{1},\dots,x_{r})=0$. 
\end{defn}

\begin{rmk} 
\begin{enumerate}[(1)]
\item $\xi_{\bf m}(D;x_{1},\dots,x_{r})$ can be interpreted as follows:
\begin{align*}
\xi_{\bf m}(D;x_{1},\dots,x_{r})=\sup_{Y^{i}_{\bullet}} \{a \ge 0 \text{ } | &\text{ } \Delta_{m_{1}a}^{-1} \times \dots \times \Delta_{m_{r}a}^{-1} \subseteq \widetilde{\Delta}_{Y^{1}_{\bullet}, \dots, Y^{r}_{\bullet}}(D)\},
\end{align*}
where the supremum is taken over all infinitesimal flags $Y^{i}_{\bullet}$ over $x_{i}$. This follows from the following easy fact: for a big $\mathbb{R}$-divisor $D$ on $X$ and any infinitesimal flags $Y^{1}_{\bullet},\dots,Y^{r}_{\bullet}$ over $x_{1},\dots,x_{r}$, we have ${\rm Conv}(\bigcup_{i=1}^{r}\Delta_{\epsilon(\lVert D \rVert;x_{i})}(D)) \subseteq \Delta_{Y^{1}_{\bullet},\dots,Y^{r}_{\bullet}}(D)$, where $\Delta_{\epsilon(\lVert D \rVert;x_{i})}(D)=(0, \dots, 0) \times \dots \times \Delta_{\epsilon(\lVert D \rVert;x_{i})}^{-1} \times \dots \times (0, \dots, 0) \subseteq \mathbb{R}^{nr}$ for each $i=1,\dots,r$. 
\item From \cite[Proposition 4.7]{KL1507}, we raise the following question:

Let $D$ be a big $\mathbb{R}$-divisor on $X$, and fix any infinitesimal flags $Y^{1}_{\bullet}, \dots, Y^{r}_{\bullet}$ over $x_{1},\dots,x_{r} \notin {\rm \bf B}_{+}(D)$. For a vector ${\bf m}=(m_{1},\dots,m_{r}) \in \mathbb{N}^{r}$, write
\begin{align*}
\xi_{{\bf m}; Y^{1}_{\bullet},\dots,Y^{r}_{\bullet}}(D;x_{1},\dots,x_{r}) \overset{\rm def}{=} \sup \{a \ge 0 \text{ } | &\text{ } \Delta_{(m_{1}a,\dots,m_{r}a)}^{-r} \subseteq \widetilde{\Delta}_{Y^{1}_{\bullet}, \dots, Y^{r}_{\bullet}}(D)\}.
\end{align*}
Then, does $\xi_{\bf m}(D;x_{1},\dots,x_{r})=\xi_{{\bf m};Y^{1}_{\bullet},\dots,Y^{r}_{\bullet}}(D;x_{1},\dots,x_{r})$ hold?
\end{enumerate}
\end{rmk}

Now, we are ready to prove our main result (Theorem \ref{thm:multi-point Seshadri constants}).

\begin{proof}[Proof of Theorem \ref{thm:multi-point Seshadri constants}]
We may assume that $D$ is a $\mathbb{Q}$-divisor since both invariants are continuous. Moreover, if $x_{i} \in {\rm \bf B}_{-}(D)$ for some $i$, the result is trivial. So we may assume that $x_{i} \notin {\rm \bf B}_{-}(D)$ for all $i$. 

First, we deal with the mono-graded case. For notational convenience, we omit the subscript $(1,\dots,1)$. The inequality $\epsilon(\lVert D \rVert;x_{1},\dots,x_{r}) \le \xi(D;x_{1},\dots,x_{r})$ follows from Lemma \ref{lem:inequality} and Proposition \ref{prop:restricted base loci}. For the reverse inequality, it is sufficient to show that $\epsilon(\lVert D \rVert;x_{1},\dots,x_{r}) \ge \alpha$, where $\alpha$ is any rational number that satisfies $\xi(D;x_{1},\dots,x_{r}) >\alpha$. Choose $t \in \mathbb{N}$ so that $t\alpha$ and $tD$ are integral. Since $\xi(D;x_{1},\dots,x_{r})>\alpha$, $\xi(mtD;x_{1},\dots,x_{r})>mt\alpha$ for any $m \in \mathbb{N}$. Also, Proposition \ref{prop:jet separation} gives that $s(K_{X}+mtD;x_{1},\dots,x_{r}) \ge mt\alpha-n$. By choosing $m \gg 0$, we may assume that $K_{X}+mtD$ is big. Then the superadditivity of $s(k(K_{X}+mtD);x_{1},\dots,x_{r})$ (see the proof of \cite[Lemma 3.7]{I}) induces the inequality 
\begin{align*}
\frac{s(k(K_{X}+mtD);x_{1},\dots,x_{r})}{k} \ge mt\alpha-n
\end{align*}
for any $k \ge 1$ and $m \gg 0$. By Lemma \ref{lem:continuity}, 
\begin{align*}
\epsilon(\lVert K_{X}+mtD \rVert;x_{1},\dots,x_{r})=\lim_{k \rightarrow \infty} \frac{s(k(K_{X}+mtD);x_{1},\dots,x_{r})}{k} \ge mt\alpha-n.
\end{align*}
Let $\beta_{m}=\frac{1}{m}$ so that the above arguments can be rephrased as 
\begin{align*}
\epsilon(\lVert \beta_{m} \cdot K_{X}+tD \rVert;x_{1},\dots,x_{r}) \ge t\alpha-n\beta_{m}
\end{align*}
for all $0<\beta_{m} \ll 1$. Again, by Lemma \ref{lem:continuity}, 
\begin{align*}
\epsilon(\lVert tD \rVert;x_{1},\dots,x_{r})=\lim_{\beta_{m} \rightarrow 0} \epsilon(\lVert \beta_{m} \cdot K_{X}+tD \rVert;x_{1},\dots,x_{r}) \ge t\alpha.
\end{align*}
Then the homogeneity of $\epsilon(-;x_{1},\dots,x_{r})$ gives the inequality $\epsilon(\lVert D \rVert;x_{1},\dots,x_{r}) \ge \xi(D;x_{1},\dots,x_{r})$.

For a general ${\bf m} \in \mathbb{N}^{r}$, the inequality $\epsilon_{\bf m}(\lVert D \rVert;x_{1},\dots,x_{r}) \le \xi_{\bf m}(D;x_{1},\dots,x_{r})$ follows from Lemma \ref{lem:inequality}. For the reverse inequality, it is also sufficient to show that $\epsilon_{\bf m}(\lVert D \rVert;x_{1},\dots,x_{r}) \ge \alpha$, where $\alpha$ is any rational number that satisfies $\xi_{\bf m}(D;x_{1},\dots,x_{r})>\alpha$. Choose $t \in \mathbb{N}$ such that $t\alpha$ and $tD$ are integral. Since $\xi_{\bf m}(D;x_{1},\dots,x_{r})>\alpha$, the inequality $\xi_{\bf m}(mtD;x_{1},\dots,x_{r})>mt\alpha$ holds for any $m \in \mathbb{N}$. That is, we have
\begin{align*}
\Delta_{(mt\alpha m_{1},\dots,mt\alpha m_{r})}^{-r} \subseteq \widetilde{\Delta}_{Y^{1}_{\bullet}, \dots, Y^{r}_{\bullet}}(mtD)
\end{align*}
for any infinitesimal flags $Y^{i}_{\bullet}$ over $x_{i}$. Let $\epsilon:=\min_{i}\{(m_{i}-1)n\}$. Since $mt\alpha m_{i} \ge n+(mt\alpha -n)m_{i}+\epsilon$, Proposition \ref{prop:jet separation} gives the inequality: 
\begin{align*}
s_{\bf m}(K_{X}+mtD;x_{1},\dots,x_{r}) \ge mt\alpha -n. 
\end{align*}
The rest of the arguments are exactly the same as in the mono-graded case, so we omit it. 
\end{proof}

Theorem \ref{thm:multi-point Seshadri constants} immediately gives the characterization of augmented base loci in terms of extended infinitesimal Okounkov bodies. 

\begin{cor} \label{cor:augmented base loci}
Let $D$ be a big $\mathbb{R}$-divisor on $X$. Then the following are equivalent. 
\begin{enumerate}[(1)]
\item $x_{1},\dots,x_{r} \notin {\rm \bf{B}}_{+}(D)$.
\item For any ${\bf m}=(m_{1},\dots,m_{r}) \in \mathbb{N}^{r}$ and every infinitesimal flags $Y^{1}_{\bullet},\dots, Y^{r}_{\bullet}$ over $x_{1},\dots,x_{r}$, there is $\xi_{\bf m}>0$, depending on ${\bf m}$, such that $\Delta^{-r}_{(m_{1} \xi_{\bf m}, \dots, m_{r} \xi_{\bf m})} \subseteq \widetilde{\Delta}_{Y^{1}_{\bullet}, \dots,Y^{r}_{\bullet}}(D)$.
\end{enumerate}
\end{cor}

The existence of an effective divisor with prescribed vanishing behavior is an interesting problem in itself, and the presence of valuative points gives an answer to this question. However, from the definition of (extended) Okounkov bodies, it is quite unclear which rational points are valuative (\cite[Definition 1.7]{KL1507}), and the problem is much more difficult when it comes to boundary points, e.g. \cite[Corollary D]{KL1507} or \cite[Question 7.3]{LM}). A by-product of our result is about the existence of valuative boundary points. 

\begin{cor} \label{cor:valuative}
Let $D$ be a big $\mathbb{R}$-divisor on $X$, $Y^{1}_{\bullet},\dots,Y^{r}_{\bullet}$ infinitesimal flags over $x_{1},\dots,x_{r}$, and ${\bf m}=(m_{1},\dots,m_{r}) \in \mathbb{N}^{r}$. If $\Delta^{-r}_{(m_{1} \xi, \dots, m_{r} \xi)} \subseteq \widetilde{\Delta}_{Y^{1}_{\bullet}, \dots, Y^{r}_{\bullet}}(D)$ for a rational $\xi \neq \epsilon_{\bf m}(\lVert D \rVert;x_{1},\dots,x_{r})$, then all vertice points of $\Delta^{-r}_{(m_{1} \xi, \dots, m_{r} \xi)}$ are valuative. 
\end{cor} 

\begin{proof}
Fix a rational number $\xi$ that satisfies the above condition. By Theorem \ref{thm:multi-point Seshadri constants}, $\Delta^{-r}_{(m_{1} \xi, \dots, m_{r} \xi)} \subseteq \widetilde{\Delta}_{Y^{1}_{\bullet}, \dots, Y^{r}_{\bullet}}(D)$ means that $\xi < \epsilon_{\bf m}(\lVert D \rVert;x_{1},\dots,x_{r})$. Then this is an immediate consequence of the proof of Lemma \ref{lem:inequality}. 
\end{proof}

\end{subsection}

\end{section}

\begin{section}{Applications} \label{section:6}

\begin{subsection}{Nagata's conjecture and volumes} \label{subsection:Nagata}

In this subsection, we see how the Nagata conjecture is interpreted in terms of convex bodies and study volumes of slices of the extended infinitesimal Okounkov bodies. In his work (\cite{N}), Nagata made the following conjecture:

\begin{conj} \label{conj:Nagata} (Nagata)
Let $\pi:{\rm Bl}_{r}(\mathbb{P}^{2}) \rightarrow \mathbb{P}^{2}$ be the blow-up of $\mathbb{P}^{2}$ at $r \ge 9$ general points. If $\pi^{*}\OO_{\mathbb{P}^{2}}(d)-\sum_{i=1}^{r}m_{i}E_{i}$ is effective, then 
\begin{align*}
d \ge \frac{1}{\sqrt{r}} \cdot \sum_{i=1}^{r}m_{i}. 
\end{align*}
\end{conj}

Taking this into account, it is clear that Conjecture \ref{conj:Nagata} is equivalent to saying that 
\begin{align*}
\epsilon(\OO_{\mathbb{P}^{2}}(1);r)=\frac{1}{\sqrt{r}}
\end{align*}
if $r \ge 9$. It still remains open, apart from the case when the number of blown up points $r$ is a square, which was settled by Nagata himself. Note that the inequality $\epsilon(\OO_{\mathbb{P}^{2}}(1);r) \le \frac{1}{\sqrt{r}}$ is obvious, so the Nagata conjecture implies that $\OO_{\mathbb{P}^{2}}(1)$ has its possible maximal positivity at $r$ general points. Some results on this conjecture can be found in \cite{B, BH1, BH2, ST, X}. 

For homogeneous linear systems, the Nagata conjecture reads as follows.
\begin{conj} \label{conj:homogeneous Nagata} (Homogeneous Nagata)
Let $\pi:{\rm Bl}_{r}(\mathbb{P}^{2}) \rightarrow \mathbb{P}^{2}$ be the blow-up of $\mathbb{P}^{2}$ at $r \ge 9$ general points. If $\pi^{*}\OO_{\mathbb{P}^{2}}(d)-\sum_{i=1}^{r}mE_{i}$ is effective, then 
\begin{align*}
d \ge \sqrt{r} \cdot m. 
\end{align*}
\end{conj}

Now, we look at the Nagata conjecture in terms of convex geometry. In this viewpoint, it is equivalent to saying that the inverted standard simplex $\Delta_{\xi}^{-r}$ can be embedded into $\widetilde{\Delta}_{Y^{1}_{\bullet},\dots,Y^{r}_{\bullet}}(\OO_{\mathbb{P}^{2}}(1))$ in its possible maximal size for all infinitesimal flags $Y^{1}_{\bullet}, \dots, Y^{r}_{\bullet}$ over $x_{1},\dots,x_{r}$. 

\begin{prop} \label{prop:Nagata}
Let $\pi:{\rm Bl}_{\{x_{1},\dots,x_{r}\}}(\mathbb{P}^{2}) \rightarrow \mathbb{P}^{2}$ be the blow-up of $\mathbb{P}^{2}$ at $r \ge 9$ general points $x_{1},\dots,x_{r}$ with exceptional divisors $E_{1},\dots,E_{r}$. Let $\mu_{r}=\sup \{m>0 \text{ } | \text{ } \pi^{*}\OO_{\mathbb{P}^{2}}(1)-m \cdot \sum_{i=1}^{r}E_{i} \text{ is big}\}$. Then the following are equivalent. 
\begin{enumerate}[(1)]
\item The Nagata conjecture holds for $r$. 
\item $\Delta_{\mu_{r}}^{-r} \subseteq \widetilde{\Delta}_{Y^{1}_{\bullet},\dots,Y^{r}_{\bullet}}(\OO_{\mathbb{P}^{2}}(1))$ for all infinitesimal flags $Y^{1}_{\bullet},\dots,Y^{r}_{\bullet}$ over $x_{1},\dots,x_{r}$. 
\end{enumerate}
Moreover, if either $(1)$ or $(2)$ holds, then 
\begin{enumerate}[(3)]
\item ${\rm vol}_{\mathbb{R}^{2}}({\widetilde{\Delta}_{Y^{1}_{\bullet},\dots,Y^{r}_{\bullet}}(\OO_{\mathbb{P}^{2}}(1))}_{\nu_{1}^{(1)}=\cdots=\nu_{1}^{(r)}, \text{ }\nu_{2}^{(1)}=\cdots=\nu_{2}^{(r)}})=\frac{1}{2}$ for all infinitesimal flags $Y^{1}_{\bullet},\dots,Y^{r}_{\bullet}$ over $x_{1},\dots,x_{r}$, where $\mathbb{R}^{2} = S_{(1,1)} \subset \mathbb{R}^{2r}$. 
\end{enumerate}
\end{prop}

\begin{proof}
Put $H=\pi^{*}\OO_{\mathbb{P}^{2}}(1)$. First, we prove  $(2) \Rightarrow (1)$. Note that $\epsilon(\OO_{\mathbb{P}^{2}}(1);r) \ge \mu_{r}$ by combining $(2)$ with Theorem \ref{thm:multi-point Seshadri constants}. Moreover, since $\mu_{r}$ is the possible maximal number $\xi$ such that $\Delta_{\xi}^{-r}$ can be embedded into $\widetilde{\Delta}_{Y^{1}_{\bullet},\dots,Y^{r}_{\bullet}}(\OO_{\mathbb{P}^{2}}(1))$ for all infinitesimal $Y^{i}_{\bullet}$ over $x_{i}$, we have $\epsilon(\OO_{\mathbb{P}^{2}}(1);r)=\mu_{r}$. For a contradiction, suppose that $\epsilon(\OO_{\mathbb{P}^{2}}(1);r)<\frac{1}{\sqrt{r}}$. Then the above argument gives that $H-\frac{1}{\sqrt{r}} \cdot \sum_{i=1}^{r} E_{i}$ is not pseudo-effective. So there exists an ample divisor $A=eH-\sum_{i=1}^{r}n_{i}E_{i}$ such that $(A.H-\frac{1}{\sqrt{r}} \cdot \sum_{i=1}^{r}E_{i})<0$. Since $A^{2}>0$, we have $e^{2}>\sum_{i=1}^{r}n_{i}^{2}$. Moreover, 
\begin{align*}
e^{2}>(\sum_{i=1}^{r}{|n_{i}|}^{2}) \cdot (\sum_{i=1}^{r} {(\frac{1}{\sqrt{r}})}^{2}) \ge {(\sum_{i=1}^{r}  \frac{|n_{i}|}{\sqrt{r}})}^{2},
\end{align*}
where $|\bullet|$ is the absolute value of $\bullet$. Since $A$ is ample, we have $e>0$ so that $e>\sum_{i=1}^{r} \frac{|n_{i}|}{\sqrt{r}} \ge \frac{1}{\sqrt{r}} \cdot \sum_{i=1}^{r} n_{i}$. Hence
\begin{align*}
0>(A.H-\frac{1}{\sqrt{r}} \cdot \sum_{i=1}^{r}E_{i})=e-\frac{1}{\sqrt{r}} \cdot \sum_{i=1}^{r} n_{i} >0,
\end{align*}
which is a contradiction. 

Furthermore, the above argument gives that $\mu_{r}=\epsilon(\OO_{\mathbb{P}^{2}}(1);r)=\frac{1}{\sqrt{r}}$ holds under the Nagata conjecture. Now, the converse direction $(1) \Rightarrow (2)$ is clear by Theorem \ref{thm:multi-point Seshadri constants}. 

Finally, we are left with checking the implication $(1) \Rightarrow (3)$. Note that the Nagata conjecture implies $\epsilon(\OO_{\mathbb{P}^{2}}(1);r)=\mu_{r}=\frac{1}{\sqrt{r}}$; in fact, since $\epsilon(\OO_{\mathbb{P}^{2}}(1);r)=\frac{1}{\sqrt{r}}$, $H-\frac{1}{\sqrt{r}} \cdot \sum_{i=1}^{r}E_{i}$ is nef but not ample. Since ${(H-\frac{1}{\sqrt{r}} \cdot \sum_{i=1}^{r}E_{i})}^{2}=0$, its nefness implies that it is not big, that is, $\frac{1}{\sqrt{r}}=\epsilon(\OO_{\mathbb{P}^{2}}(1);r)=\mu_{r}$. By Theorem \ref{thm:multi-point Seshadri constants} and the Nagata conjecture, 
\begin{align*}
{\rm vol}_{\mathbb{R}^{2}}({\widetilde{\Delta}_{Y^{1}_{\bullet},\dots,Y^{r}_{\bullet}}(\OO_{\mathbb{P}^{2}}(1))}_{\nu_{1}^{(1)}=\cdots=\nu_{1}^{(r)}, \text{ }\nu_{2}^{(1)}=\cdots=\nu_{2}^{(r)}}) \ge \frac{{\sqrt{r \cdot {\epsilon(\OO_{\mathbb{P}^{2}}(1);r)}^{2}}}^{2}}{2}=\frac{1}{2}.
\end{align*}
Moreover, Lemma \ref{lem:upper bound} gives that 
\begin{align*}
{\rm vol}_{\mathbb{R}^{2}}({\widetilde{\Delta}_{Y^{1}_{\bullet},\dots,Y^{r}_{\bullet}}(\OO_{\mathbb{P}^{2}}(1))}_{\nu_{1}^{(1)}=\cdots=\nu_{1}^{(r)}, \text{ }\nu_{2}^{(1)}=\cdots=\nu_{2}^{(r)}}) \le \frac{{\sqrt{r \cdot {\mu_{r}}^{2}}}^{2}}{2}=\frac{1}{2}.
\end{align*}
Hence, we have $(3)$ as wanted. 
\end{proof}

From $(3)$ in Proposition \ref{prop:Nagata}, we propose Conjecture \ref{conj:Shin} that not only provides a great help in dealing with the multi-weight moving Seshadri constants, but also connects the extended Okounkov bodies with intersection theory. For this conjecutre, we have a positive answer when $m_{1}=\dots=m_{r}=1$. For notational convenience, we write 
\begin{align*}
{\widetilde{\Delta}_{Y^{1}_{\bullet},\dots,Y^{r}_{\bullet}}(D)}_{(m_{1},\dots,m_{r})}={\widetilde{\Delta}_{Y^{1}_{\bullet},\dots,Y^{r}_{\bullet}}(D)}_{\frac{\nu_{1}^{(1)}}{m_{1}}=\cdots=\frac{\nu_{1}^{(r)}}{m_{r}}, \dots, \text{ } \frac{\nu_{n}^{(1)}}{m_{1}}=\cdots=\frac{\nu_{n}^{(r)}}{m_{r}}}.
\end{align*} 

\begin{prop} \label{prop:mono-graded}
For any infinitesimal flags $Y^{1}_{\bullet},\dots,Y^{r}_{\bullet}$ over general points $x_{1},\dots,x_{r} \in X$, 
\begin{enumerate}[(1)]
\item ${\rm vol}_{\mathbb{R}^{n}}({\widetilde{\Delta}_{Y^{1}_{\bullet},\dots,Y^{r}_{\bullet}}(D)}_{(1,\dots,1)})=\frac{(\sqrt{r})^{n-2}}{n!} \cdot {\rm vol}_{X}(D)$, and
\item ${\rm vol}_{\mathbb{R}^{2}}({\widetilde{\Delta}_{Y^{1}_{\bullet},\dots,Y^{r}_{\bullet}}(D)}_{(m_{1},\dots,m_{r})}) \le \frac{1}{2} \cdot {\rm vol}_{X}(D)$ when ${\rm dim}X=2$. 
\end{enumerate}
\end{prop}

\begin{proof}
(1) Let ${\rm pr}_{j}$ be the $j$-th projection map defined in Note \ref{note:projection}, and let $\Delta_{j}(D)$ be the multipoint Okounkov body of $D$ at $x_{j}$ (\cite{AT}). Since $x_{1},\dots,x_{r}$ are general, it is easy to see that $F(\Delta_{j}(D))={\rm pr}_{j}({\widetilde{\Delta}_{Y^{1}_{\bullet}, \dots,Y^{r}_{\bullet}}(D)}_{(1,\dots,1)})$, where $F:\mathbb{R}^{n} \rightarrow \mathbb{R}^{n}, (y_{1},\dots,y_{r}) \mapsto (y_{1}+\dots+y_{n}, y_{1},\dots,y_{n-1})$, and all the volumes of $F(\Delta_{j}(D))$ coincide for all $j$. Moreover, note that
\begin{align*}
{\rm vol}_{\mathbb{R}^{n}}({\widetilde{\Delta}_{Y^{1}_{\bullet}, \dots,Y^{r}_{\bullet}}(D)}_{(1,\dots,1)})&=(\sqrt{r})^{n} \cdot {\rm vol}_{\mathbb{R}^{n}}({\rm pr}_{j}({\widetilde{\Delta}_{Y^{1}_{\bullet}, \dots,Y^{r}_{\bullet}}(D)}_{(1,\dots,1)}) \\
&=(\sqrt{r})^{n} \cdot {\rm vol}_{\mathbb{R}^{n}}(F(\Delta_{j}(D))).
\end{align*}
Since $F$ is bijective and its Jacobian is $1$, \cite[Theorem 3.3.2, p.96]{EG} implies that ${\rm vol}_{\mathbb{R}^{n}}(\Delta_{j}(D))={\rm vol}_{\mathbb{R}^{n}}(F(\Delta_{j}(D)))$. Now the result is an immediate consequence of \cite[Theorem A]{AT}. 

(2) Let $W_{k,j}^{(m_{1},\dots,m_{r})}:=\{s \in H^{0}(X,kD)-\{0\} \text{ } | \text{ } \frac{1}{m_{j}}\nu^{p_{j}}(s)<\frac{1}{m_{i}}\nu^{p_{i}}(s) \text{ for all } i \neq j\}$, where $\nu^{p_{j}}(s)$ is the leading term exponent at $p_{j}$ with respect to a total additive order on $\mathbb{Z}^{n}$ (cf. \cite{AT}). We define a convex body $\Delta_{j}^{(m_{1},\dots,m_{r})}(D):=\overline{\bigcup_{k \ge 1} \{\frac{\nu^{p_{j}}(s)}{k}, s \in W_{k,j}^{(m_{1},\dots,m_{r})}\}} \subseteq \mathbb{R}^{n}$. Note that $F(\Delta_{j}^{(m_{1},\dots,m_{r})}(D))={\rm pr}_{j}({\widetilde{\Delta}_{Y^{1}_{\bullet}, \dots,Y^{r}_{\bullet}}(D)}_{(m_{1},\dots,m_{r})})$, and that 
\begin{align*}
{\rm vol}_{\mathbb{R}^{2}}({\widetilde{\Delta}_{Y^{1}_{\bullet}, \dots,Y^{r}_{\bullet}}(D)}_{(m_{1},\dots,m_{r})})&=(\frac{\sqrt{m_{1}^{2}+\dots+m_{r}^{2}}}{m_{j}})^{2} \cdot {\rm vol}_{\mathbb{R}^{2}}({\rm pr}_{j}({\widetilde{\Delta}_{Y^{1}_{\bullet}, \dots,Y^{r}_{\bullet}}(D)}_{(m_{1},\dots,m_{r})}) \\
&=\frac{m_{1}^{2}+\dots+m_{r}^{2}}{m_{j}^{2}} \cdot {\rm vol}_{\mathbb{R}^{2}}(\Delta_{j}^{(m_{1},\dots,m_{r})}(D)).
\end{align*}
Since it is easy to see that $\sum_{j=1}^{r} {\rm vol}_{\mathbb{R}^{2}}(\Delta_{j}^{(m_{1},\dots,m_{r})}(D)) \le \frac{1}{2} \cdot {\rm vol}_{X}(D)$, we are done. 
\end{proof}

\begin{subsection} {${\overline{\rm Eff}({\rm Bl}_{r}(S))}_{\mathbb{R}}$ and ${\overline{\rm Nef}({\rm Bl}_{r}(S))}_{\mathbb{R}}$ for a surface $S$.}

We observe a general relation between ${\overline{\rm Eff}({\rm Bl}_{r}(S))}_{\mathbb{R}}$ and ${\overline{\rm Nef}({\rm Bl}_{r}(S))}_{\mathbb{R}}$ for a surface $S$. 

\begin{defn}
Let $\pi:{\rm Bl}_{\{x_{1},\dots,x_{r}\}}(X) \rightarrow X$. The Nakayama constant of a big $\mathbb{R}$-divisor $D$ on $X$ at $x_{1},\dots,x_{r}$ is the real number
\begin{align*}
\mu(D;x_{1},\dots,x_{r})=\sup \{a \ge 0 \text{ } | \text{ } \pi^{*}D-a \cdot \sum_{i=1}^{r}E_{i} \text{ is big}\}.
\end{align*} 
\end{defn}

\begin{prop} \label{prop:general relation}
Let $S$ be a smooth projective surface, $L$ an ample $\mathbb{R}$-divisor on $S$, and $x_{1},\dots,x_{r} \in S$. Then 
\begin{align*}
\mu(L;x_{1},\dots,x_{r})-\sqrt{{\mu(L;x_{1},\dots,x_{r})}^{2}-\frac{1}{r} \cdot L^{2}} \le \epsilon(L;x_{1},\dots,x_{r}) \le \frac{L^{2}}{r \cdot \mu(L;x_{1},\dots,x_{r})}. 
\end{align*}
In particular, $\epsilon(L;x_{1},\dots,x_{r})=\sqrt{\frac{L^{2}}{r}}$ if and only if $\mu(L;x_{1},\dots,x_{r})=\sqrt{\frac{L^{2}}{r}}$. 
\end{prop}

\begin{proof}
Let $\epsilon=\epsilon(L;x_{1},\dots,x_{r})$ and $\mu=\mu(L;x_{1},\dots,x_{r})$. Note that a similar argument to the proof of \cite[Proposition 4.2]{KL1411} says that ((R.H.S) in Proposition \ref{prop:surfaces}) $\cap S_{(1,\dots,1)}$ is contained in the convex hull of the set $\{{\bf 0}, {\rm v}_{1}^{(\sqrt{r}\epsilon, \dots, \sqrt{r} \epsilon)}+{\rm v}_{2}^{(\sqrt{r}\epsilon, \dots, \sqrt{r} \epsilon)}, {\rm v}_{1}^{(\sqrt{r}\mu, \dots, \sqrt{r} \mu)}, {\rm v}_{1}^{(\sqrt{r}\mu, \dots, \sqrt{r} \mu)}+{\rm v}_{2}^{(\sqrt{r}\epsilon, \dots, \sqrt{r} \epsilon)}\}$ in $S_{(1, \dots, 1)}$, denoted by $\Delta$. Thus ${\widetilde{\Delta}_{Y^{1}_{\bullet},\dots,Y^{r}_{\bullet}}(L)}_{(1,\dots,1)} \subseteq \Delta$ in $S_{(1,\dots,1)}$. By Proposition \ref{prop:mono-graded}, 
\begin{align*}
\frac{L^{2}}{2}={\rm vol}_{\mathbb{R}^{2}}({\widetilde{\Delta}_{Y^{1}_{\bullet},\dots,Y^{r}_{\bullet}}(L)}_{(1,\dots,1)}) \le {\rm vol}_{\mathbb{R}^{2}}(\Delta)=\frac{1}{2}r\epsilon^{2}+r\epsilon(\mu-\epsilon).
\end{align*}
So $\mu-\sqrt{\mu^{2}-\frac{1}{r} \cdot L^{2}} \le \epsilon$. Furthermore, for very general infinitesimal flags $Y^{1}_{\bullet},\dots,Y^{r}_{\bullet}$ over $x_{1},\dots,x_{r}$, it is easy to see that the convex hull of the set $\{{\bf 0}, {\rm v}_{1}^{(\sqrt{r}\epsilon, \dots, \sqrt{r} \epsilon)}+{\rm v}_{2}^{(\sqrt{r}\epsilon, \dots, \sqrt{r} \epsilon)}, {\rm v}_{1}^{(\sqrt{r}\mu, \dots, \sqrt{r} \mu)}\}$ in $S_{(1, \dots, 1)}$. So we have $\frac{L^{2}}{2} \ge \frac{r}{2} \cdot \epsilon \mu$, i.e. $\epsilon \le \frac{L^{2}}{r\mu}$. 
\end{proof}

\end{subsection}

\begin{rmk}
If Conjecture \ref{conj:Shin} holds for arbitrary $(m_{1},\dots,m_{r}) \in \mathbb{N}^{r}$, then we can induce a general relation between ${\overline{\rm Eff}({\rm Bl}_{r}(S))}_{\mathbb{R}}$ and ${\overline{\rm Nef}({\rm Bl}_{r}(S))}_{\mathbb{R}}$ by using $\epsilon_{\bf m}(L;x_{1},\dots,x_{r})$ and the multi-weight version of Nakayama constants. 
\end{rmk}

\end{subsection}

\begin{subsection} {Irrationality of Seshadri constant, ${\overline{\rm Eff}({\rm Bl}_{s}(\mathbb{P}^{2}))}_{\mathbb{R}}$ and ${\overline{\rm Nef}({\rm Bl}_{s}(\mathbb{P}^{2}))}_{\mathbb{R}}$}
From Proposition \ref{prop:general relation}, we can draw a stronger conclusion than \cite{DKMS, HH} under a weaker assumption. 

Let $r=1$ and $S={\rm Bl}_{s}(\mathbb{P}^{2})$. Then Proposition \ref{prop:general relation} reads as follows: 
\begin{align*}
\mu(L;x) -\sqrt{\mu(L;x)^{2}-L^{2}} \le \epsilon(L;x) \le \frac{L^{2}}{\mu(L;x)}.
\end{align*}

\begin{defn} \label{defn:standard form}
A line bundle $F=dH-\sum_{i=1}^{s}m_{i}E_{i}$ on ${\rm Bl}_{s}(\mathbb{P}^{2})$ is in standard form if 
\begin{enumerate}[(i)]
\item $m_{1} \ge \dots \ge m_{s} \ge 0$, and 
\item $d \ge m_{1}+m_{2}+m_{3}$. 
\end{enumerate}
\end{defn}

From SHGH Conjecture (\cite{DKMS}), we present the following conjecture:

\begin{conj} \label{conj:standard}
Let $D$ be a line bundle in standard form on ${\rm Bl}_{s}(\mathbb{P}^{2})$ for $s \ge 1$. If $D^{2}<0$, then $D$ is not effective. 
\end{conj}

Recall the $(-1)$-curve conjecture:
\begin{conj}((-1)-curve conjecture) \label{conj:(-1)}
A prime divisor $C$ on ${\rm Bl}_{s}(\mathbb{P}^{2})$ with $C^{2}<0$ is a $(-1)$-curve. 
\end{conj}

\begin{lem} \label{lem:(-1) and standard}
The $(-1)$-curve conjecture implies Conjecture \ref{conj:standard}. 
\end{lem}

\begin{proof}
Assume Conjecture \ref{conj:(-1)}. Let $D$ be a line bundle in standard form on ${\rm Bl}_{s}(\mathbb{P}^{2})$ with $D^{2}<0$. Suppose that $D$ is effective. Then we can consider its Zariski decomposition $D=P+N$, where $P$ and $N$ are effective $\mathbb{Q}$-divisors. By choosing a sufficiently divisible $m>0$, we may assume that $mP$ and $mN$ are integral. So we may let $mN=C_{1}+\cdots+C_{t}$, where $C_{i}$ is a prime divisor on ${\rm Bl}_{s}(\mathbb{P}^{2})$ for $1 \le i \le t$. (Note that some of them may be the same.) 

Since $D$ is effective and $P$ is nef, $(D.mD-C_{1}-\cdots-C_{t})=(D.mP) \ge 0$ so that $0>m(D^{2}) \ge (D.C_{1}+\cdots+C_{t})$. Hence $(D.C_{i})<0$ for some $i$. Moreover, the negative-definiteness of the intersection matrix of $N$ implies that $C_{i}^{2}<0$ for all $i$. Since $C_{i}$'s are prime with $C_{i}^{2}<0$, they are all $(-1)$-curves by Conjecture \ref{conj:(-1)}. However, since any line bundles in standard form have non-negative intersection with any $(-1)$-curves, $(D.C_{i}) \ge 0$ for all $i$. This contradicts the above argument that $(D.C_{i})<0$ for some $i$. Hence $D$ is not effective. 
\end{proof}

\begin{thm} \label{thm:irrationality}
\cite[Main Theorem]{DKMS} and \cite[Theorem 2.4]{HH} hold, assuming only Conjecture \ref{conj:standard}. More precisely, assume Conjecture \ref{conj:standard} for $s+1$ points, and let $L=dH-\sum_{i=1}^{s}m_{i}E_{i}$ be an ample line bundle on ${\rm Bl}_{s}(\mathbb{P}^{2})$ satisfying 
\begin{enumerate}[(i)]
\item $m_{1} \ge \cdots \ge m_{s} \ge 0$, and
\item $m_{1}+m_{2}+m_{3} \le d < \frac{(m_{1}+m_{2})^{2}+\sum_{i=1}^{s} m_{i}^{2}}{2(m_{1}+m_{2})}$. 
\end{enumerate}
Then $\epsilon(L;x)=\sqrt{L^{2}}$ for a general point $x \in {\rm Bl}_{s}(\mathbb{P}^{2})$. 
\end{thm}

\begin{proof}
Let $\phi:{\rm Bl}_{s+1}(\mathbb{P}^{2}) \rightarrow {\rm Bl}_{s}(\mathbb{P}^{2})$ be the blow-up of ${\rm Bl}_{s}(\mathbb{P}^{2})$ at a general point $x \in {\rm Bl}_{s}(\mathbb{P}^{2})$ with an exceptional divisor $E$. Let $\alpha$ be a rational number such that $\alpha>\sqrt{d^{2}-\sum_{i=1}^{s}m_{i}^{2}}=\sqrt{L^{2}}$ and that $\lvert \alpha-\sqrt{d^{2}-\sum_{i=1}^{s}m_{i}^{2}} \rvert$ is sufficiently small. We claim that $L_{m,\alpha}=mdH-m \sum_{i=1}^{s}m_{i}E_{i}-m\alpha E$ is not effective for all sufficiently large and divisible $m \gg 0$. Since $L_{m,\alpha}^{2}<0$ and the Conjecture \ref{conj:standard} holds for $s+1$, it is enough to show that $L_{m,\alpha}$ is in standard form. If $\sqrt{d^{2}-\sum_{i=1}^{s}m_{i}^{2}}<m_{3}$, then we are done by choosing $\sqrt{d^{2}-\sum_{i=1}^{s}m_{i}^{2}}<\alpha<m_{3}$. So assume that $\sqrt{d^{2}-\sum_{i=1}^{s}m_{i}^{2}} \ge m_{3}$. Now, it suffices to show that $d>m_{1}+m_{2}+\sqrt{d^{2}-\sum_{i=1}^{s}m_{i}^{2}}$: if so, we may choose a rational $\alpha$ satisfying $d>m_{1}+m_{2}+\alpha>m_{1}+m_{2}+\sqrt{d^{2}-\sum_{i=1}^{s}m_{i}^{2}}$. However, it holds by our assumption $d < \frac{(m_{1}+m_{2})^{2}+\sum_{i=1}^{s} m_{i}^{2}}{2(m_{1}+m_{2})}$, which proves the claim. 

Now, it is a direct consequence of Proposition \ref{prop:general relation}. 
\end{proof}

\begin{rmk}
The condition on $d$ in Theorem \ref{thm:irrationality} implies that $s \ge 9$: in order to guarantee the existence of $d$, the inequality 
\begin{align*}
m_{1}+m_{2}+m_{3} < \frac{(m_{1}+m_{2})^{2}+\sum_{i=1}^{s} m_{i}^{2}}{2(m_{1}+m_{2})}
\end{align*}
should hold. Note that it is equivalent to the inequality $2m_{1}m_{2}+2m_{2}m_{3}+2m_{3}m_{1}<m_{3}^{2}+\cdots+m_{s}^{2}$. Since $m_{1} \ge \cdots \ge m_{s}$, it implies the inequality $m_{3}^{2}+\cdots+m_{8}^{2} \le 2m_{1}m_{2}+2m_{2}m_{3}+2m_{3}m_{1}<m_{3}^{2}+\cdots+m_{s}^{2}$. Thus we have $s \ge 9$. 
\end{rmk}

One advantage of Theorem \ref{thm:irrationality} is that the proof is simple compared to the importance of its conclusion: in fact, assumptions for the full Conjecture \ref{conj:standard} are not required for the existence of irrational Seshadri constants. Moreover, it provides a useful tool in computing lower bounds of Seshadri constants on general rational surfaces under Conjecture \ref{conj:standard}. For example, the following result holds:

\begin{cor} \label{cor:homogeneous}
Let $s \ge 9$, $L=dH-c \cdot \sum_{i=1}^{s}E_{i}$ an ample line bundle on ${\rm Bl}_{s}(\mathbb{P}^{2})$, and $x \in {\rm Bl}_{s}(\mathbb{P}^{2})$ a general point. Suppose Conjecture \ref{conj:standard} for quasi-homogeneous forms on $s+1$ points. 
\begin{enumerate}[(1)]
\item If $\frac{c}{d} \ge \frac{4}{r+4}$, then $\epsilon(L;x)=\sqrt{L^{2}}$. 
\item If $\frac{c}{d} < \frac{4}{r+4}$, then $\epsilon(L;x) \ge d-2c$. 
\end{enumerate}
\end{cor}

\begin{proof}
Keeping the notation of Theorem \ref{thm:irrationality}, note that $m_{1}=\cdots=m_{s}=c$. Since $s \ge 9$, $m_{1}+m_{2}+m_{3}=3c \le \sqrt{s} \cdot c <d$ since $\epsilon(\OO_{\mathbb{P}^{2}}(1);s) \le \frac{1}{\sqrt{s}}$ and $L$ is ample. Hence the inequality $m_{1}+m_{2}+m_{3} \le d$ holds automatically in this case. Moreover, 
\begin{align*}
\frac{(m_{1}+m_{2})^{2}+\sum_{i=1}^{s} m_{i}^{2}}{2(m_{1}+m_{2})}=\frac{c(s+4)}{4}.
\end{align*}
So $d<\frac{(m_{1}+m_{2})^{2}+\sum_{i=1}^{s} m_{i}^{2}}{2(m_{1}+m_{2})}$ is equivalent to $\frac{c}{d}>\frac{4}{s+4}$. Thus $(1)$ is a consequence of Theorem \ref{thm:irrationality} when $\frac{c}{d}>\frac{4}{s+4}$. 

If $\frac{c}{d}=\frac{4}{s+4}$, $\sqrt{L^{2}}=d-2c$, so we are left with checking $(2)$ when $\frac{c}{d} \le \frac{4}{s+4}$. Choose a rational $d'$ satisfying 
\begin{align*}
c \cdot \sqrt{s+1}<d'<(\frac{s+4}{4})c \le d. 
\end{align*}
It is possible since $s \ge 9$. Then $L$ can be written as $(d-d')H+(d'H-c \cdot \sum E_{i})$, so let $L'=d'H-c \cdot \sum E_{i}$. By our choice of $d'$ with \cite[Theorem 2]{ST}, $L'$ is an ample $\mathbb{Q}$-divisor on ${\rm Bl}_{s}(\mathbb{P}^{2})$. Choose a sufficiently divisible $m>0$ so that $mL'$ is integral. Then $(1)$ implies that 
\begin{align*}
\epsilon(mL';x)=\sqrt{m^{2}L'^{2}}=m \sqrt{(d')^{2}-sc^{2}},
\end{align*}
that is, $L-\sqrt{(d')^{2}-sc^{2}}$ is nef. Moreover, $\epsilon((d-d')H;x)=d-d'$ holds so that 
\begin{align*}
\epsilon(dH-c \cdot \sum_{i=1}^{s}E_{i};x) \ge (d-d')+\sqrt{(d')^{2}-sc^{2}}.
\end{align*}
Let $f(d')=(d-d')+\sqrt{(d')^{2}-sc^{2}}$. Since $f'(d')=-1+\frac{d'}{\sqrt{(d')^{2}-sc^{2}}}>0$, $f(d')$ is an increasing function on $d'$. Since $d'$ is a rational number smaller than $(\frac{s+4}{4})c$ and the nef cone is closed, 
\begin{align*}
\epsilon(dH-c \cdot \sum_{i=1}^{s}E_{i};x) \ge (d-(\frac{s+4}{4})c)+\sqrt{(\frac{s+4}{4})^{2}c^{2}-sc^{2}}=d-2c
\end{align*}
as desired. 
\end{proof}

As another by-product of Proposition \ref{prop:general relation}, we obtain the complete characterization of nef $\mathbb{R}$-divisors lying on the boundary of ${\overline{\rm Eff}({\rm Bl}_{s}(\mathbb{P}^{2}))}_{\mathbb{R}}$.

\begin{cor}
Let $L=dH-\sum_{i=1}^{s}m_{i}E_{i}$ be a $\mathbb{R}$-divisor lying on the boundary of ${\overline{\rm Eff}({\rm Bl}_{s}(\mathbb{P}^{2}))}_{\mathbb{R}}$. Then $L$ is nef if and only if $L^{2}=0$. In particular, homogeneous Nagata's conjecture implies Nagata's conjecture.
\end{cor}

\begin{proof}
One direction is clear, so we focus on the reverse one. We may assume that $m_{1} \ge \dots \ge m_{s}>0$. By Proposition \ref{prop:general relation}, it is sufficient to show that $dH-\sum_{i=2}^{s}m_{i}E_{i}$ is nef. For its nefness, we will use \cite[Theorem 2.1]{H12}. 

Note that 
\begin{align*}
d^{2}-\frac{t+3}{t+2}(m_{2}^{2}+\dots+m_{t+1}^{2})&=m_{1}^{2}-\frac{1}{t+2}(m_{2}^{2}+\dots+m_{t+1}^{2})+m_{t+2}^{2}+\dots+m_{s}^{2} \\
&=\frac{1}{t+2}(m_{1}^{2}+\dots+m_{1}^{2})-\frac{1}{t+2}(m_{2}^{2}+\dots+m_{t+1}^{2})+m_{t+2}^{2}+\dots+m_{s}^{2}>0
\end{align*}
for any $2 \le t \le s-1$. So the condition (4) in \cite[Theorem 2.1]{H12} is checked. 

For (3) in \cite[Theorem 2.1]{H12}, note that
\begin{align*}
&3 \sqrt{m_{1}^{2}+\dots+m_{s}^{2}} > 2m_{2}+m_{3}+\dots+m_{8} \\ 
\Longleftrightarrow \text{ }&9(m_{1}^{2}+\dots+m_{s}^{2})>4m_{2}^{2}+m_{3}^{2}+\dots+m_{8}^{2}+4m_{2}(m_{3}+\dots+m_{8})+ 2m_{3}(m_{4}+\dots+m_{8})+ \\
\text{ }&2m_{4}(m_{5}+\dots+m_{8})+\dots+2m_{7}m_{8} \\ 
\Longleftrightarrow \text{ }&9m_{1}^{2}+5m_{2}^{2}+8m_{3}^{2}+\dots+8m_{t}^{2}-4m_{2}(m_{3}+\dots+m_{8})-2m_{3}(m_{4}+\dots+m_{8})-\dots-2m_{7}m_{8}>0.
\end{align*}
Using $m_{1} \ge \dots m_{s} \ge 0$, 
\begin{align*}
&9m_{1}^{2}+5m_{2}^{2}+8m_{3}^{2}+\dots+8m_{t}^{2}-4m_{2}(m_{3}+\dots+m_{8})-2m_{3}(m_{4}+\dots+m_{8})-\dots-2m_{7}m_{8} \\
&\ge 8m_{1}^{2}+4m_{3}^{2}+4m_{4}^{2}+\dots+4m_{8}^{2}-2m_{3}(m_{4}+\dots+m_{8})-\dots-2m_{7}m_{8} \\
&\ge 7m_{1}^{2}+3m_{4}^{2}+\dots+3m_{8}^{2}-2m_{4}(m_{5}+\dots+m_{8})-\dots-2m_{7}m_{8} \\
&\ge 6m_{1}^{2}+2m_{5}^{2}+\dots+2m_{8}^{2}-2m_{5}(m_{6}+m_{7}+m_{8})-\dots-2m_{7}m_{8}\\
&\ge 5m_{1}^{2}+m_{6}^{2}+m_{7}^{2}+m_{8}^{2}-2m_{6}(m_{7}+m_{8})-2m_{7}m_{8} \\
&\ge 4m_{1}^{2}-2m_{7}m_{8} >0, 
\end{align*}
which satisfies (3) in \cite[Theorem 2.1]{H12}. 

We are left with checking (1) $d \ge m_{2}+m_{3}$ and (2) $2d \ge m_{2}+\dots+m_{6}$. For (1), suppose that $d < m_{2}+m_{3}$. For $\alpha>0$, consider $(L-\alpha(H-E_{2}-E_{3}))^{2}=\alpha (2(m_{2}+m_{3}-d)-\alpha)$. By chossing a sufficiently small $0<\alpha \ll 1$, $(L-\alpha(H-E_{2}-E_{3}))^{2}>0$ since $m_{2}+m_{3}-d>0$. So $L-\alpha(H-E_{2}-E_{3})=(d-\alpha)H-m_{1}E_{1}-(m-\alpha)E_{2}-(m-\alpha)E_{3}-\sum_{i=4}^{s}m_{i}E_{i}$ is big by Riemann-Roch theorem. However, since $L$ and $H-E_{2}-E_{3}$ lie on the boundary of $({\overline{\rm Eff}({\rm Bl}_{s}(\mathbb{P}^{2}))}_{\mathbb{R}})$, $L-\alpha (H-E_{2}-E_{3})$ cannot be big for any $\alpha >0$. Thus $d \ge m_{2}+m_{3}$. 

Finally, for (2), suppose that $2d<m_{2}+\dots+m_{6}$. Again, consider $2H-E_{2}-\dots-E_{6}$ and $L-\alpha(2H-E_{2}-\dots-E_{6})$ for $\alpha>0$. The argument similar to the above leads to a contradiction. Hence $2d \ge m_{2}+\dots+m_{6}$. 
\end{proof}

\end{subsection}

\end{section}

\bibliographystyle{abbrv}

\begin{thebibliography}{00}


\bibitem{BDHKKSS} Th. Bauer, S. Di Rocco, B. Harbourne, M. Kapustka, A. Knutsen, W. Syzdek, and T. Szemberg, A primer on Seshadri constants. Interactions of classical and numerical algebraic geometry, 33–-70, Contemp. Math., 496, Amer. Math. Soc., Providence, RI, 2009. 

\bibitem{BKS} T. Bauer, A. K\" uronya, and T. Szemberg, Zariski chambers, volumes, and stable base loci, J. Reine Angew. Math. 576 (2004), 209–-233. 


\bibitem{B} P. Biran, Constructing new ample divisors out of old ones. Duke Math. J. 98 (1999), 113--135. 



\bibitem{CHPW} S. R. Choi, Y. Hyun, J. Park, and J. Won, Asymptotic base loci via Okounkov bodies. Adv. Math. 323 (2018), 784-–810. 

\bibitem{CPW1} S. R. Choi, J. Park, and J. Won, Okounkov bodies associated to pseudoeffective divisors. J. London Math. Soc. (2) 98 (2018), 170--195. 

\bibitem{CPW2} S. R. Choi, J. Park, and J. Won, Okounkov bodies associated to pseudoeffective divisors II. Taiwanese J. Math. 21 (2017), no. 3, 601–-620. 





\bibitem{CLS} D. Cox, J. Little, and H. Schenck, {\em Toric varieties.} 
Graduate Studies in Mathematics, 124. American Mathematical Society, Providence, RI, 2011. xxiv+841 pp. ISBN: 978-0-8218-4819-7 

\bibitem{JPD} J. P. Demailly, Singular Hermitian metrics on positive line bundles. Complex algebraic varieties (Bayreuth, 1990), 87–-104, 
Lecture Notes in Math., 1507, Springer, Berlin, 1992. 

\bibitem{D} Y. Deng, Transcendental Morse inequality and generalized Okounkov bodies. Algebr. Geom. 4 (2017), no. 2, 177–-202.





\bibitem{DKMS} M. Dumnicki, A. K\"uronya, C. Maclean, and T. Szemberg, Rationality of Seshadri constants and the Segre-Harbourne-Gimigliano-Hirschowitz conjecture. Adv. Math. 303 (2016), 1162–-1170. 

\bibitem{ELMNP2} L. Ein, R. Lazarsfeld, M. Musta\c t\v a, M. Nakamaye, and M. Popa, Asymptotic invariants of base loci. Ann. Inst. Fourier (Grenoble) 56 (2006), no. 6, 1701–-1734. 

\bibitem{ELMNP1} L. Ein, R. Lazarsfeld, M. Musta\c t\v a, M. Nakamaye, and M. Popa, Restricted volumes and base loci of linear series. Amer. J. Math. 131 (2009), no. 3, 607-–651. 

\bibitem{EG} L. C. Evans and R. F. Gariepy, {\em Measure theory and fine properties of functions.} Revised edition. Textbooks in Mathematics. CRC Press, Boca Raton, FL, 2015. xiv+299 pp. ISBN: 978-1-4822-4238-6 



\bibitem{H12} K. Hanumanthu, Positivity of line bundles on general blow ups of $\mathbb{P}^{2}$. J. Algebra 461 (2016), 65–-86. 


\bibitem{BH1} B. Harbourne, On Nagata's Conjecture. J. Alg. 236 (2001), 692--702. 

\bibitem{BH2} B. Harbourne, Seshadri constants and very ample divisors on algebraic surface. J. Reine Angew. Math. 559 (2003), 115--122. 

\bibitem{HR} B. Harbourne and J. Ro\'e, Discrete behavior of Seshadri constants on surfaces. J. Pure Appl. Algebra 212 (2008), no. 3, 616–-627.


\bibitem{H} R. Hartshorne, Ample subvarieties of algebraic varieties. Lect. Notes in Math., vol. 156, Springer 1970.

\bibitem{HH} K. Hanumanthu and B. Harbourne, Single point Seshadri constants on rational surfaces. J. Algebra 499 (2018), 37-–42. 

\bibitem{I} A. Ito, Okounkov bodies and Seshadri constants. Adv. Math. 241 (2013), 246-–262. 

\bibitem{J} S.-Y. Jow, Okounkov bodies and restricted volumes along very general curves. Adv.Math. 223 (2010), no. 4, 1356–-1371.

\bibitem{KK} K. Kaveh and A. G. Khovanskii, Newton-Okounkov bodies, semigroups of integral points, graded algebras and intersection theory. Ann. of Math. (2) 176 (2012), no. 2, 925–-978.


\bibitem{KL1509} A. K\"uronya and V. Lozavanu, A Reider-type theorem for higher syzygies on abelian surfaces. arXiv:1509.08621.

\bibitem{KL1703} A. K\"uronya and V. Lozavanu, Geometric aspects of Newton-Okounkov bodies. arXiv:1703.09980. 

\bibitem{KL1507} A. K\"uronya and V. Lozovanu, Infinitesimal Newton-Okounkov bodies and jet separation. Duke Math. J. 166 (2017), no. 7, 1349–-1376.

\bibitem{KL1411} A. K\"uronya and V. Lozovanu, Local positivity of linear series on surfaces. Algebra Number Theory 12 (2018), no. 1, 1–-34. 

\bibitem{KL1506} A. K\"uronya and V. Lozovanu, Positivity of line bundles and Newton-Okounkov bodies, Doc. Math. 22 (2017), 1285-–1302. 

\bibitem{KLM} A. K\"uronya, V. Lozovanu, and C. Maclean, Convex bodies appearing as Okounkov bodies of divisors. Adv. Math. 229 (2012), no. 5, 2622-–2639. 

\bibitem{L} R. Lazarsfeld, {\em Positivity in Algebraic Geometry I, II}. Ergebnisse der Mathematik und ihrer Grenzgebiete. 3. Folge. A Series of Modern Surveys in Mathematics [Results in Mathematics and Related Areas. 3rd Series. A Series of Modern Surveys in Mathematics], 48. Springer-Verlag, Berlin, 2004. 

\bibitem{LM} R. Lazarsfeld and M. Musta\c t\v a, Convex bodies associated to linear series. Ann. Sci. \'Ec. Norm. Sup\'er. (4) 42 (2009), no. 5, 783–-835. 


\bibitem{N} M. Nagata, On the 14--th problem of Hilbert. Amer. J. Math. 81 (1959), 766--772. 

\bibitem{Na} M. Nakamaye, Base loci of linear series are numerically determined, Trans. Amer. Math. Soc. 355, no. 2 (2002), 551–-566.

\bibitem{O1} A. Okounkov, Brunn-Minkowski inequality for multiplicities.
Invent. Math. 125 (1996), no. 3, 405–-411.

\bibitem{O2} A. Okounkov, Why would multiplicities be log-concave? in The orbit method in geometry and physics. Progr. Math., 213, Birkh\"auser Boston, Boston, MA, 2003.  


\bibitem{R} J. Ro\'e, Local positivity in terms of Newton-Okounkov bodies. Adv. Math. 301 (2016), 486–-498. 

\bibitem{S} J. Shin, Higher syzygies on abelian surfaces. arXiv:1608.06052. 



\bibitem{ST} T. Szemberg and H. Tutaj-Gasi\'n ska, General blow-ups of the projective plane. Proc. Amer. Math. Soc. 130 (9) (2002) 2515--2524. 

\bibitem{AT} A. Trusiani, Multipoint Okounkov bodies, arXiv:1804.02306. 


\bibitem{X} G. Xu, Curves in $\mathbb{P}^{2}$ and symplectic packings. Math. Ann. 299 (1994), 609--613. 

\end{thebibliography}

\end{document}